\documentclass{amsart}
\usepackage[utf8]{inputenc}

\usepackage{eurosym}
\usepackage{amsfonts}
\usepackage{amsmath}
\usepackage{amssymb}
\usepackage{graphicx}
\usepackage{color}
\usepackage[dvipsnames]{xcolor}
\usepackage{mathrsfs}
\usepackage{lipsum}
\usepackage{soul}
\usepackage{bm}
\usepackage{bbm}
\usepackage{comment}
\usepackage{enumerate}
\usepackage{framed}
\usepackage{tabularx}
\usepackage{siunitx}
\usepackage{booktabs}

\usepackage{appendix}
\usepackage{hyperref}
\usepackage{cleveref}
\usepackage{amsthm}
\usepackage{mathtools}
\usepackage{stmaryrd} 

\usepackage[normalem]{ulem}

\newtheorem{theorem}{Theorem}[section]

\newtheorem{assumption}[theorem]{Assumption}

\newtheorem{definition}[theorem]{Definition}

\newtheorem{lemma}[theorem]{Lemma}

\newtheorem{proposition}[theorem]{Proposition}

\theoremstyle{definition}
\newtheorem{remark}[theorem]{Remark}

\numberwithin{equation}{section}

\newcommand{\one}{\mathbbm{1}}
\newcommand{\ep}{\varepsilon}

\newcommand{\bE}{\mathbb{E}}
\newcommand{\bN}{\mathbb{N}}
\newcommand{\bR}{\mathbb{R}}

\newcommand{\bT}{\mathbb{T}}

\newcommand{\cC}{\mathcal{C}}
\newcommand{\cD}{\mathcal{D}}
\newcommand{\cF}{\mathcal{F}}

\newcommand{\cH}{\mathcal{H}}
\newcommand{\cL}{\mathcal{L}}

\newcommand{\cK}{\mathcal{K}}
\newcommand{\cN}{\mathcal{N}}

\newcommand{\cX}{\mathcal{X}}
\newcommand{\cY}{\mathcal{Y}}

\newcommand{\du}{\, \mathrm{d}u}
\newcommand{\ds}{\, \mathrm{d}s}

\newcommand{\D}{\, \mathrm{d}}

\newcommand{\E}{\mathrm{e}}
\newcommand{\BV}{\mathrm{BV}}
\newcommand{\VIX}{{\rm VIX}}

\newcommand{\tom}{\widetilde{\Omega}}
\newcommand{\tbd}{\widetilde{\bm{d}}}

\newcommand{\wpar}{\widehat{\partial}}
\newcommand{\tpar}{\widetilde{\partial}}
\newcommand{\tcL}{\widetilde{\cL}}

\newcommand{\dr}{\,\mathrm{d}r}

\newcommand{\half}{\frac{1}{2}}

\newcommand{\etab}{{\bar{\eta}}}

\newcommand{\omegab}{\bar{\omega}}
\newcommand{\gammab}{\bar{\gamma}}

\definecolor{ocean}{rgb}{0,0.1,0.6}
\definecolor{imperialGreen}{RGB}{2,137,59}

\newcommand{\notthis}[1]{}

\newcommand{\Wh}{\widehat{W}}

\newcommand{\norm}[1]{\left\lVert#1\right\rVert}
\newcommand{\abs}[1]{\left\lvert#1\right\rvert}

\let\oldtocsection=\tocsection
\let\oldtocsubsection=\tocsubsection
\let\oldtocsubsubsection=\tocsubsubsection
\renewcommand{\tocsection}[2]{\hspace{0em}\oldtocsection{#1}{#2}}
\renewcommand{\tocsubsection}[2]{\hspace{1em}\oldtocsubsection{#1}{#2}}
\renewcommand{\tocsubsubsection}[2]{\hspace{2em}\oldtocsubsubsection{#1}{#2}}

\usepackage{setspace}
\usepackage[margin=3cm]{geometry}
\usepackage[textsize=small,textwidth=25mm]{todonotes}

\usepackage{booktabs}
\usepackage[tableposition=above]{caption}
\usepackage{pifont}
\usepackage{multirow}
\usepackage{shuffle}

\usepackage{amsaddr} 

\begin{document}

\title{A path-dependent PDE solver based on signature kernels}

\author{Alexandre Pannier$^{*}$ and Cristopher Salvi$^{\dagger}$}
\address{$^{*}$Université Paris Cité, Laboratoire de Probabilités, Statistiques et Modélisation\\
$^{\dagger}$Imperial College London, Department of Mathematics and Imperial-X}
\email{pannier@lpsm.paris}
\email{c.salvi@imperial.ac.uk}


\begin{abstract}
    We develop a kernel-based solver for path-dependent PDEs (PPDEs) along with a convergence theory.
    Our numerical scheme leverages signature kernels, a recently introduced class of kernels on path-space. Specifically, we solve an optimal recovery problem by approximating the solution of a PPDE with an element of minimal norm in the signature reproducing kernel Hilbert space constrained to satisfy the PPDE at a finite collection of collocation paths. In the linear case, we show that the optimisation has a unique closed-form solution expressed in terms of signature kernel evaluations at the collocation paths. Under strict assumptions, we prove consistency of the proposed scheme, guaranteeing convergence to the PPDE solution as the number of collocation points increases. Finally, several numerical examples are presented, in particular in the context of option pricing under rough volatility. Our numerical scheme constitutes a valid alternative to the ubiquitous Monte Carlo methods.
\end{abstract}

\maketitle




\section{Introduction}

\subsection{Path-dependent PDEs}
The Feynman-Kac formula asserts that solutions to backward Kolmogorov equations have a probabilistic representation 
as~$u_0(t,x)=\bE[\phi(Y_T)\lvert Y_{t}=x]$, for~$0\le t\le T,\, x\in\bR^d$, where~$Y$ is an underlying Markov process.
If~$Y$ models the evolution of the price of a financial asset, then this expectation corresponds to the price of a European-type derivative written on this asset. 
In its seminal paper, Dupire~\cite{dupire2022functional} introduces a first generation of path-dependent PDEs (PPDEs) which solution is represented as~$u_1(t,\gamma) = \bE[\phi(Y_{[0,T]}) \lvert Y_{[0,t]}=\gamma]$
where~$\phi$ acts on the whole past of~$Y$ and~$\gamma$ is a path. This expectation represents the price of a path-dependent derivative. Rough volatility models, and more generally stochastic Volterra processes, are not Markov processes hence neither formulation applies. The recent framework of~\cite{viens2019martingale} tackles precisely this class of processes, for which the path-dependence is inherent to the dynamics. We illustrate this with an example to highlight the differences with the aforementioned cases.
Let us consider a one factor rough volatility model 
\begin{equation}\label{eq:MainModel}
     X_t = X_0 + \int_{0}^t \sqrt{\psi_s(\Wh_s)} \D B_s -\half \int_0^t \psi_s(\Wh_s)\ds, \qquad \Wh_t  :=  \int_0^t K(t,r) \D W_r, 
\end{equation}
where~$X$ represents the log-price, $B$ and~$W$ are correlated Brownian motions and generate their natural filtration~$(\cF_t)_{t\in[0,T]}$, the kernel~$K$ is square integrable and typically taken to be~$K(t,r)=(t-r)^{H-\half}$ with~$H\in(0,1/2)$. In this situation $\Wh$ is a Gaussian Volterra process, which is neither a semimartingale nor a Markov process, that is why classical tools are not suitable. For~$s\ge t$, we make the orthogonal decomposition~$\Wh_s=\Theta^t_s +I^t_s$ where~$\Theta^{t}_s =  \int_0^{t} K(s,r) \,\D W_r $ is~$\cF_t$-measurable while~$I^t_s=\int_{t}^s K(s,r) \,\D W_r$ is independent from~$\cF_t$. As a consequence, one recovers the Markovian structure as follows:
\begin{align*}
    \mathbb{E}[\phi(X_T)|\cF_t]
    & = \mathbb{E}\left[\phi\left(X_t + \int_t^T \sqrt{\psi_s(\Wh_s)} \D B_s - \half\int_t^T \psi_s(\Wh_s) \ds \right) \Big\lvert\cF_t \right]\\
    & = \mathbb{E}\left[\phi \left(X_t + \int_t^T \sqrt{\psi_s(\Theta^t_s + I^t_s)} \D B_s - \half \int_t^T \psi_s(\Theta^t_s + I^t_s)\ds \right)\Big\lvert X_t, \Theta^t_{[t,T]}\right]\\
    & = \bE\left[\phi(X_T) | X_t, \Theta^t_{[t,T]}\right].
\end{align*}
This exhibits the path-dependence of the conditional expectation~$\bE[\phi(X_T)\lvert\cF_t]$ (corresponding to an option price with underlying asset~$\exp(X)$) and entails there is a measurable map~$u$ such that~$u(t,x,\gamma) = \bE\left[\phi(X_T) | X_{t}=x, \Theta^{t}_{[t,T]}=\gamma_{[t,T]}\right]$. Building on Viens and Zhang's functional Itô formula, further works~\cite{bonesini2023rough,wang2022path} study the associated path-dependent PDE of second generation of the form
\begin{equation}\label{eq:mainPDE_intro}
\left\{\hspace{-.4cm}
\begin{array}{rl}
    &\cL u(t,x,\gamma)= 0, \\
    &u(T,x,\gamma)=\phi(x),
\end{array}
\right.
\end{equation}
for all $t\in[0,T],x\in\bR,\gamma\in \cC([0,T];\bR)$, and where the linear operator $\cL$ is a combination of derivatives of all inputs. The pathwise derivatives are defined in Section \ref{subsec:functions_paths} while important well-posed examples can be found in Section~\ref{subsec:examples}. This paper aims at solving this type of PPDEs.

Apart from the realm of rough volatility, stochastic Volterra processes have appeared in the modeling of electricity prices and turbulence~\cite{barndorff2014assessing,bennedsen2022rough,corcuera2013asymptotic,mishura2023gaussian}, in principal-agent problems~\cite{abi2022gaussian}, climate modeling~\cite{eichinger2020sample,sonechkin1998climate} and in asymptotic studies motivated by physical phenomena~\cite{gehringer2022functional,hairer2005ergodicity,li2022slow}. More broadly, fractional Brownian motions and related processes are present in a variety of situations where intertemporal correlation is observed, starting with the seminal paper~\cite{mandelbrot1968fractional}.

While the PPDE representation finds other applications (e.g. weak error rates in \cite{bonesini2023rough}), this article is concerned with their numerical computation. The solution to a PPDE is a function of a path (and other finite-dimensional inputs). Contrary to the wide literature on numerical schemes for high-dimensional PDEs, we are here interested in purely infinite-dimensional inputs that we will treat as such instead of discretising them. Despite the central place of such PPDEs in mathematical finance, numerical methods pertaining to them are scarse; we refer to Section~\ref{subsec:literature} for an overview of the literature. Our motivation is then twofold: design a numerical solver for such PDEs and overcome the technical challenges of learning a function on paths. 



\subsection{Signature kernels}

Kernel methods are a well-established class of algorithms that are the key components of several machine learning models such as support vector machines (SVMs). The central idea of kernel methods is to lift the (possibly unstructured) input data to a (possibly infinite-dimensional) Hilbert space by means of a nonlinear \emph{feature map} defined in terms of a kernel. The advantage of doing so lies in the fact that a non-linear optimisation task in the original space becomes linear once lifted to the feature space. Thus, the solution of the original optimisation problem can often be expressed entirely in terms of evaluations of the kernel on training points by means of celebrated representer theorems. The kernel can often be evaluated efficiently with no reference to the feature map, a property known as \emph{kernel trick}. The selection of an effective kernel will usually be a task-dependent problem, and this challenge is particularly difficult when the data is sequential. Signature kernels  \cite{kiraly2019kernels, salvi2021signature} are a class of \emph{universal kernels} on sequential data which have received attention in recent years
 thanks to their efficiency in handling path-dependent problems \cite{lemercier2021distribution, salvi2021higher, cochrane2021sk, lemercier2021siggpde, cirone2023neural, issa2023non, horvath2023optimal, manten2024signature}.

\subsection{Kernels vs neural networks for PDEs} In recent years, PDE solvers based on neural networks and kernel methods have been introduced as alternatives to classical finite difference and finite elements schemes. The most famous examples of neural PDE solvers are the Deep Galerkin method (DGM) \cite{sirignano2018dgm} and Physics informed neural networks (PINNs) \cite{raissi2019physics}. These techniques essentially parameterise the solution of the PDE as a neural network which is then trained using the PDE and its boundary conditions as loss function. More recently, DGM and PINNs have been extended to path-dependent PDEs by \cite{sabate2020solving, saporito2021path, jacquier2023deep}. PDE solvers based on kernel methods have been investigated in various works \cite{cockayne2016probabilistic, raissi2017machine, raissi2018numerical, chen2021solving, chen2023sparse}, but the development of kernel methods for path-dependent PDEs has remained open. PDE solvers based on kernel methods have the potential for considerable advantages over neural PDE solvers, both in terms of theoretical analysis and numerical implementation. In effect, the theoretical analysis of neural PDE solvers is often limited to density or universal approximation results aimed at showing the existence of a network of a requisite size achieving
a certain error rate, without guaranteeing whether this network is computable in practice via stochastic gradient descent. In contrast, as we will demonstrate in this work, kernel-based PDE solvers are provably convergent under certain conditions and amenable to rigorous numerical analysis based on a rich arsenal of functional analytic results from the kernel literature. In addition to their theoretical appeal, kernel PDE solvers also enjoy concrete numerical benefits over neural PDE solvers. For linear PDEs such as those treated in this paper, for example, training reduces to solving a linear system of equations, whose stable numerical solutions are readily obtained with standard linear algebra routines. This stands in contrast to neural PDE solvers, where training involves non-convex optimization, often unstable in practice, and highly sensitive to hyperparameter choices.

\subsection{Monte Carlo methods for pricing under rough volatility} The notion of \emph{rough volatility} has recently disrupted stochastic modelling in quantitative finance \cite{bayer2023rough}, leaving no one indifferent in the field. In this setting, the instantaneous volatility of assets is modelled as a fractional Brownian motion with small Hurst parameter as in~\eqref{eq:MainModel} or as a stochastic Volterra process. This feature has been shown empirically to be consistent with historical market trajectories and to capture important stylised facts of equity options in a parsimonious way. However, this new paradigm comes with an increased computational cost compared to classical techniques. With the notable exception of the rough
Heston model, the absence of Markovianity of volatility prevents any pricing tools other than Monte Carlo simulations. However, the new promises--–for estimation and calibration--–of these rough volatility models have encouraged deep and fast innovations in numerical methods for pricing, in particular the now standard Hybrid scheme \cite{bennedsen2017hybrid}. The complexity of a Monte Carlo scheme scales roughly as $\mathcal{O}(h^{-2} N d^2)$, where $N$ is the number of sample paths, $h$ is the discretisation step and $d$ is the dimension of the process. This complexity can be further reduced to $\mathcal{O}((h\log(h))^{-1} N d^2)$ with hybrid schemes  \cite{bennedsen2017hybrid}. In this paper, we propose a simulation-free alternative to Monte Carlo pricing by constructing a kernel-based solver for the corresponding PPDE. The complexity of the resulting solver scales as $\mathcal{O}(h^{-2} \hat{N}d^2)$ where this time $\hat{N}$ is the number of collocation paths. This cost can be further reduced to~$\mathcal{O}(h^{-1} \hat{N} d^2)$ if the operations are carried out on a GPU, see \cite{salvi2021signature} for details. As our numerical experiments demonstrate, a transparent comparison of complexities between our approach and Monte Carlo for pricing under rough volatility can only be achieved by obtaining error rates for the convergence of the approximations, which we intend to explore as future work. 

\subsection{Contributions}
In this paper we design a novel  numerical solver for PPDEs based on signature kernel and accompanied by a convergence theory. 
Our numerical scheme does not use any classical probabilistic representation or any other pre-existing solver; it constitutes a valid alternative to the ubiquitous Monte Carlo methods, in particular in the context of option pricing  under rough volatility. Our kernel-based solver does not require additional ``training'' costs: the estimation of the Gram matrix is performed once, offline, and can then be reused for multiple configurations, payoffs, and evaluation points $(t,x,\gamma)$. By contrast, while Monte Carlo methods also allow the reuse of sample paths across different payoffs, a new simulation is required for each evaluation point $(t,x,\gamma)$. Moreover, Monte Carlo is only applicable when the PPDE admits a probabilistic representation. In settings where such a representation is unavailable, Monte Carlo cannot be used, whereas our solver remains applicable.

The hurdles in the theoretical analysis are many-fold and their resolution contributes to the development of the theory of signature kernels. Challenges include the translation of the PPDE framework into one amenable for the use of signature kernels, the proof of  existence and uniqueness of a global minimum in the optimisation problem as well as an analytic form of the solution by means of a representer theorem, and a consistency result based on the adaptation of convergence results from~\cite{chen2021solving} to this infinite-dimensional setting via a compactness argument for functions on path spaces (Theorem~\ref{th:MainCvg} and Lemma~\ref{lemma:compact_H}).  Finally, several numerical examples are presented in Section~\ref{sec:numerics} and the code is open-sourced at \url{https://github.com/crispitagorico/sigppde}.


\subsection{Related literature}\label{subsec:literature}

Several deep learning and signature-based approaches have been proposed in recent years. In \cite{sabate2020solving}, the authors use a probabilistic approach based on the martingale representation theorem to design an objective function for solving time-discretised PPDEs via a parametric model combining RNNs and signatures. The martingale representation theorem is also central in~\cite{fang2023neural}, where the authors make use of continuous-time RNNs developed in \cite{morrill2021neural}. In both these papers, the signature has to be truncated for the implementation. The authors of~\cite{saporito2021path} extend the deep Galerkin framework proposed by \cite{sirignano2018dgm} using an LSTM architecture at each time step. The only paper applicable to Volterra processes~\cite{jacquier2023deep} develops a backward discretisation algorithm using neural networks,  drawing inspiration from the deep learning methodology introduced in \cite{han2018solving}.  We note that none of the aforementioned papers present theoretical results on the proposed numerical schemes and they all rely on the training of neural networks to approximate the solution. 
Aloof from the learning paradigm, convergence guarantees of monotone schemes can be found in~\cite{ren2017convergence} for the viscosity solutions to first generation PPDEs.
The present paper is, to our knowledge, the first to propose a numerical scheme for second generation PPDEs with a convergence theory. 

\subsection{Outline}
Section~\ref{sec:PPDEs} describes the rigorous framework of PPDEs and how we intent to approximate them. This is followed in Section~\ref{sec:kernels} by an introduction to signature kernels and the definitions we will need in the remainder of the paper. The main theoretical analysis and results are presented in Section~\ref{sec:theory} while the numerical experiments can be found in Section~\ref{sec:numerics}. We give an outlook on future research in Section~\ref{sec:outlook}. Finally, Appendix~\ref{sec:appendix} gathers some technical proofs. 

\section{Path-dependent PDEs}\label{sec:PPDEs}

As the example from the introduction suggests, in situations where the underlying process is fractional, such as fractional Brownian motion, rough volatility models or stochastic Volterra processes, conditional expectations are functions of paths. In this section, we will present several instances where these functions are solutions to well-posed path-dependent PDEs. Beforehand, we must set the mathematical scene in a rigorous way. We introduce the notations and definitions needed to formalise equations such as~\eqref{eq:mainPDE_intro}. This framework enables to present central examples in subsection~\ref{subsec:examples}, then in subsection~\ref{subsec:bounded_variation} we adapt this setup to the signature realm and finally outline the numerical method is subsection~\ref{subsec:method}.

\subsection{Path spaces} \label{subsec:path_spaces}
Let us set $d,e\in\bN_0:=\bN\cup\{0\}$, a finite time horizon $T>0$ and an interval~$\bT:=[0,T]$. Throughout this paper we deal with continuous paths~$\gamma:\bT\to\bR^e$ and their perturbed c\`adl\`ag version~$\gamma+\ep\eta\one_{[t,T]}$ for which we need to introduce a number of spaces and norms. When the interval~$\bT$ and state space~$\bR^e$ are clear from the context we will often omit them. 
The space of continuous functions from~$\bT$ to~$\bR^e$, denoted~$\cC(\bT,\bR^e)$, is equipped with the supremum norm~$\norm{\gamma}_{0;\bT}:=\sup_{t\in\bT} \abs{\gamma_t}$, unless stated otherwise. The notation~$\abs{\cdot}$ refers to the Euclidean norm in one or multiple dimensions. The space of càdlàg paths from~$\bT$ to~$\bR^e$ is denoted~$D(\bT,\bR^e)$ while for all~$t\in\bT$ we will make use of the space of càdlàg paths which are continuous on~$[t,T]$, that is
$$
\cD_t:= \Big\{\gamma\in D(\bT,\bR^e): \gamma \big\lvert_{[t,T]}\in \cC([t,T],\bR^e)\Big\}.
$$
This space allows to consider paths of the form~$\gamma+\ep\eta\one_{[t,T]}$ where~$\gamma,\eta\in \cC(\bT,\bR^e)=\cD_0$ and~$\ep>0$ which in turn permits the definition of pathwise derivatives on this space.
In parallel, for~$p\ge1$, we also leverage the $p$-variation seminorm which is defined for~$0\le t<t'\le T$ and a path~$\gamma:[t,t']\to\bR^e$ as
$$
\abs{\gamma}_{p-var;[t,t']}:=\left( \sup_{(t_i)_i\in\pi} \sum_{i} \abs{\gamma_{t_{i+1}}-\gamma_{t_i}}^p\right)^{1/p},
$$
where~$\pi$ is the set of all partitions of the form~$t=t_0<t_1<\cdots<t_n=t'$ for some~$n\in\bN$. The associated norm is~$\norm{\gamma}_{p-var;[t,t']}=\abs{\gamma_t}+\abs{\gamma}_{p-var;[t,t']}$. This gives rise to the spaces~$\cC^{p-var}([t,t'],\bR^e)$  of 
continuous paths with one strictly monotone coordinate \footnote{This is a technical assumption needed to ensure injectivity of the signature and consequently the point separating property assumption needed to establish universality of the signature kernel by Stone-Weierstrass arguments. More generally one could work with equivalence classes obtained by quotienting $\cC$ by the so-called \emph{tree-like equivalence relation} $x \sim_\tau y \iff S(x) = S(y)$; with the assumption that the paths have one strictly monotone coordinate, each equivalence class collapses to a singleton, which is the path itself. The monotone coordinate is usually taken to be time.} and finite $p$-variation norm. For each such path~$\gamma$ we denote by~$\gamma^0$ the strictly monotone coordinate and by~$\check{\gamma}$ the $\bR^{e-1}$-valued path without the strictly monotone coordinate.

In the case~$p=1$, the space~$\cC^{1-var}([t,t'],\bR^e)$ is called the bounded variation space and also denoted $\BV([t,t'],\bR^e)$.
Of particular interest to us are the spaces~$\cC^{p-var}_t:=\cC^{p-var}([t,T],\bR^e)$ and~$\BV_t:=\cC^{1-var}_t$ for some~$t\in\bT$.
While the path space~$\cD_t$ in the original theory of~\cite{viens2019martingale} allows for one jump, this would lead to severe complications for the application of signature kernels. We show in the next subsection how to adapt our framework to avoid the introduction of jumps.
Furthermore, we want to equip the space~$\BV_0$ with a topology that renders the signature map continuous (see Section~\ref{sec:kernels} for the details). 
We could choose any~$1\le p$-variation topology for this task, however the necessity to characterise compact sets of~$\BV_0$ later on forces us to consider~$p>1$. This task involves the introduction of Hölder spaces~$\cC^{\alpha}([t,t'],\bR^e)$, with~$\alpha\in(0,1)$ endowed with the $\alpha$-Hölder norm
$$
\norm{\gamma}_{\alpha;[t,t']} := \abs{\gamma_t}+\abs{\gamma}_{\alpha;[t,t']}:=\abs{\gamma_t}+ \sup_{0\le s<t\le T} \frac{\abs{\gamma_t-\gamma_s}}{\abs{t-s}^\alpha}.
$$
If the interval~$[t,t']$ is clear from the context, for instance if~$\gamma\in\BV([t,T],\bR^e)$, we omit to write it as a subscript.
Having in mind to characterise compact subsets of~$\cC^{p-var}$, we follow~\cite{cuchiero2023global} and intersect it with a Hölder space: let~$\cC^{p-var,\alpha}([t,t'],\bR^e):=\cC^{p-var}([t,t'],\bR^e)\cap \cC^\alpha([t,t'],\bR^e)$ be the space of $\alpha$-Hölder continuous paths with finite~$p$-variation. The corresponding norm is~$\norm{\gamma}_{p-var,\alpha;[t,t']} :=  \abs{\gamma}_{\alpha;[t,t']} + \abs{\gamma}_{p-var;[t,t']}$. It is then proven in~\cite[Theorem A.8]{cuchiero2023global} that the embedding~$\cC^{p-var,\alpha} \hookrightarrow \cC^{p'-var,\alpha'}$ is compact for all~$1\le p<p'\le\infty$ and~$0\le \alpha'<\alpha< 1$ such that~$\alpha p<1$ and~$\alpha'p'<1$. In particular, bounded subsets of~$\cC^{1-var,\alpha}$ are compact with respect to the~$p$-variation topology for all~$\alpha\in(0,1)$ and~$p>1$, see also~\cite[Proposition 5.30]{friz2010multidimensional}.

The solutions to the path-dependent PDEs of interest are functionals of time, space and paths, hence we define suitable space and distance, borrowed from~\cite{viens2019martingale}:
\begin{equation}\label{eq:def_Omegatilde}
\tom = \Big\{(t,x,\gamma)\in\bT\times\bR^d\times \cD_t\Big\}, \qquad \tbd((t,x,\gamma),(\bar{t},\bar{x},\gammab)):= \abs{t-\bar{t}}+\abs{x-\bar{x}} + \norm{\gamma-\gammab}_{0;\bT}.
\end{equation}
Note that~$\gamma$ lives in~$\cD_t$ which depends directly on~$t$. We will denote by $\Omega$  the state space with~$\BV_0$ instead of $\cD_t$ and such that paths are constant on~$[0,t]$ (except for the strictly monotone coordinate), that is
\begin{equation}\label{eq:def_Omega}
\Omega := \big\{(t,x, \gamma)\in\bT\times \bR^d\times \BV_0: \check{\gamma}_s=\check{\gamma}_0 \;\text{for all  } s\le t\big\}.
\end{equation}
The analogous distance is then
\begin{equation}\label{eq:distance_d}
\bm{d}((t,x,\gamma),(\bar{t},\bar{x},\gammab)):= \abs{t-\bar{t}}+\abs{x-\bar{x}} + \norm{\gamma-\gammab}_{p-var;\bT}.
\end{equation}
Notice that if~$B$ is a compact subset of~$(\BV_0,\norm{\cdot}_{p-var})$ then for any reals~$a<b$, the set~$\bT\times[a,b]^d\times B$ is a compact subset of~$\Omega$ with respect to~$\bm{d}$.


\subsection{Functions of continuous paths} \label{subsec:functions_paths}
This subsection presents the framework of~\cite{viens2019martingale}.
Let~$\cC(\tom):=\cC(\tom, \bR)$ denote the set of real-valued functions~$u:\tom\to\bR$ continuous under $\tbd$.
For~$u\in\cC(\tom)$, define the right time derivative
$$
\partial_t u(t,x,\gamma):= \lim_{\ep\downarrow0}\frac{u(t+\ep,x,\gamma)-u(t,x,\gamma)}{\ep},
$$
for all~$t\in[0,T)$, provided the limit exists. By convention we set $\partial_t u(T,x,\gamma)=0$. The derivative $\partial_{x} u(t,x,\gamma)$ is defined in the same way for all~$(t,x,\gamma)\in\tom$.
 We also introduce the directional derivative~$\tpar_{\gamma}^\eta u(t,x,\gamma)$ with respect to~$\gamma\in \cD_0$, in the direction of~$\eta\in \cC([t,T],\bR^e)$:
\begin{equation}
     u(t,x,\gamma+\ep\eta\one_{[t,T]})-u(t,x,\gamma) =   \tpar_{\gamma}^\eta u(t,x,\gamma) + o(\norm{\eta\one_{[t,T]}}_{\bT}).
    \label{eq:Frechet_def}
\end{equation}
Notice that the path is only perturbed on the interval~$[t,T]$, hence the necessity of a path space allowing for a jump at $t$.
The  operator~$\eta\mapsto\tpar_\gamma^\eta u$ is linear on $\cC([t,T],\bR^e)$ and a Fréchet derivative; moreover if it exists it is equal to the Gateaux derivative.
We observe crucially that both the Fréchet and Gateaux derivative are operators acting on functions from~$\tom$ to~$\bR$; in particular the time dependence is essential.
We define higher derivatives as linear operators on~$(\cC([t,T],\bR^e))^{n}$, in particular the second derivative is a bilinear operator on~$\cC([t,T],\bR^e)\times \cC([t,T],\bR^e)$ and we denote it~$\tpar_\gamma^{\eta\etab} u(t,x,\omega)$ for~$\eta,\etab\in \cC([t,T],\bR^e)$.

Let~$n,n_1,n_2,n_3\in\bN_0$.
We denote the higher derivatives~$\partial_t^{n_1},\,\partial_x^{n_2}$ and~$\tpar_\gamma^{\bm{\eta}}$, for~$\bm{\eta}=(\eta_1,\cdots,\eta_{n_3})\in (\cC([t,T],\bR^e))^{n_3}$; the order of the directional derivative is implied by the dimension of~$\bm{\eta}$. We say that a function~$u\in \cC(\widetilde{\Omega})$ belongs to~$\cC^{\cN}(\widetilde{\Omega})$ if the linear operator~$\partial_t^{n_1}\,\partial_x^{n_2} \, \tpar_\gamma^{\bm{\eta}}  
$ exists and is continuous with respect to~$\tbd$ for all~$(n_1,n_2,n_3)\in \cN \subset \bN^3_0$. The examples of interest arise from Kolmogorov equations, hence only the first order time derivative appears and is never combined with the spatial and pathwise derivatives, while the latter are present up to order two. The corresponding subset is~$\cN= \{(n_1,n_2,n_3)\in \bN^3_0: n_1\le 1, n_2=n_3=0, \text{ or  } n_1=0, n_2+n_3\le 2\}$.

We intend to deal with fractional processes with singular kernels of the form $K(s,t)=(s-t)^{H-\half}\one_{s>t}$ with~$H\in(0,\half)$. 
Taken as a function of~$s\in[t,T]$, these kernels turn out to be the direction of the pathwise derivative but lie outside of~$\cC(\bT,\bR)$, hence an approximation procedure is required. For~$0\le t<s\le T$, we define the path $K^t(s):=K(s,t)$ and its truncated version $K^{\delta,t}(s):=K(s\wedge (t+\delta),t)$, for~$\delta>0$. The directional derivative for such a singular path is then defined, if it exists, as
\begin{align}\label{eq:singular_der}
     \tpar_{\gamma}^{K^t} u(t,x,\gamma) &:= \lim_{\delta\downarrow0}  \tpar_{\gamma}^{K^{\delta,t}} u(t,x,\gamma),
\end{align}
and similarly for the higher derivatives.
The generalisation of the space~$\cC^\cN$ to this type of derivative is more involved and requires regularity estimates tailored to the speed of explosion of the kernel. For precise definitions, we will point towards the appropriate references in the examples below. 
We are now ready to state the equations we wish to solve.


\subsection{The target PDEs} \label{subsec:examples}
This paper's focus is on approximating functionals~$u$ which only depend on the path over $[t,T]$, said differently there exist~$\tilde{u}$ such that~$u(t,x,\gamma)=\tilde{u}(t,x,\gamma\lvert_{[t,T]})$. This is the case of all the examples presented in this section.
These functionals are solutions to linear PPDEs taking the general form for all~$(t,x,\gamma)\in\tom$
\begin{equation}\label{eq:mainPDE}
\def\arraystretch{1.5}
    \left\{\hspace{-.4cm}
\begin{array}{rl}
&\tcL_t u(t,x,\gamma)= g(t,x,\gamma),\\
&u(T,x,\gamma)=\phi(T,x,\gamma),
\end{array}
\right.
\end{equation}
where~$g,\phi:\tom\to\bR$ are measurable functions and the linear operator $\tcL_t$ is a combination of derivatives of all arguments.
More precisely, there exist~$\cN\subset\bN^3_0$, $a_n\in\bR$ for all~$n\in\cN$ and~$\bm{\eta}\in (\cC([t,T],\bR^e))^{n_3}$ such that
\begin{equation}\label{eq:Def_tcL}
    \tcL_t := \sum_{n=(n_1,n_2,n_3)\in\cN} a_n \partial_t^{n_1}\,\partial_x^{n_2} \, \tpar_\gamma^{\bm{\eta}},
\end{equation}
where the sum is taken over multi-indices~$n\in\cN$. We emphasise on the~$t$ dependence which is important because the directional derivative is a linear operator on $\cC([t,T],\bR^e)$. We need this notation instead of the more standard multi-index one in order to emphasise on the order of each derivative (in our applications $n_1\le 1$, $n_2+n_3\le 2$) and also on the direction~$\bm{\eta}$. 

More precisely, this paper will focus on three examples of parabolic second order PDEs which arise from the Feynman-Kac formula. The solution to these PPDEs thus have a probabilistic representation, a crucial advantage to assess the accuracy of our method. However we emphasise that the numerical scheme does not use this a-priori knowledge as in supervised learning tasks. Much like deep Galerkin or PINN algorithms~\cite{sirignano2018dgm,raissi2019physics}, our method can be deployed to solve PPDEs for which no numerical method is known.

Our examples take as underlying a Gaussian Volterra process~$\Wh$ defined as $\Wh_t  :=  \int_0^t K(t,r) \D W_r$ for some  kernel~$K\in L^2(\bT^2,\bR)$. The key object, already introduced in the introduction, is the double index process, for~$0\le t\le s$, $\Theta^t_s = \bE[\Wh_s\lvert\cF_t]= \int_0^{t} K(s,r) \,\D W_r$. In alignment with the prolific rough volatility literature, we will mostly focus on the specification~$K(t,r)=(t-r)^{H-1/2}\one_{r\in[0,t)}$ for $H\in(0,1)$ but other choices are possible, see for instance \cite[Assumption 2.6]{bonesini2023rough}. The current setup (introduced in~\cite{viens2019martingale}), as well as our numerical method, allow for path-dependent payoffs. Since~$\Wh$ is a fractional process, the path-dependence is inherent and transfers to the PDE even for state-dependent terminal conditions (payoffs). For clarity of exposition and because well-posedness results are scarcer in the path-dependent case, we will however stick to state-dependent payoffs in our numerical experiments.

In each of the three subsequent subsections, we present
(i) the model (ii) the value function~$u$ and (iii) the PPDE, in that order. We point to the suitable references for well-posedness and for more details. We emphasise that all three examples are particular cases of~\cite[Proposition 2.14]{bonesini2023rough}; this result ensures, under general assumptions, existence and uniqueness of a classical solution to the PPDE associated to a Volterra process. 
In all three examples, the kernel's singular behaviour around $0$ must be dominated by a power-law~$t\mapsto t^{H-\half},H\in(0,\half)$ and functions are assumed three times differentiable with either exponential (for those applying to a Gaussian variable) or polynomial growth (for the others).

\subsubsection{Fractional Brownian motion}\label{subsec:example_fbm}
In this first toy example we consider test functions of $(\Wh_s)_{s\in[t,T]}$ which corresponds to the state space $\tom$ with $d=0$ and $e=1$. Let us define the function
\begin{align}\label{eq:uFBM}
u(t,\gamma) :=&\, \bE\left[\phi(\Wh_T) + \int_t^T f(s,\Wh_s) \ds \Big|\Theta^t=\gamma \right] \\
=&\, \bE\left[ \phi\left(\gamma_T+\int_t^T K(T,s)\D W_s\right) + \int_t^T f\left(s,\gamma_s+\int_t^s K(t,r)\D W_r\right) \ds \right].
\end{align}
In particular, we observe that $u(t,\Theta^t) = \bE\left[\phi(\Wh_T) + \int_t^T f(s,\Wh_s) \ds \Big|\cF_t\right].$
Note that the path-dependence of the conditional expectation is contained in the path $\Theta^t$ because the payoff only acts on $\Wh_s$ for $s\in[t,T]$.

Based on \cite[Theorem 4.1]{viens2019martingale}, under some standard regularity assumptions on~$\phi,f$ and~$K$, $u$ is a solution to the path-dependent heat equation for all~$(t,\gamma) \in \tom$
\begin{equation}\label{eqn:PPDEfBM}
\def\arraystretch{1.5}
    \left\{\hspace{-.4cm}
\begin{array}{rl}
&\partial_{t}u(t,\gamma) + \half  \tpar^{K^tK^t}_{ \gamma} u(t,\gamma) + f(t,\gamma) = 0, \\ &u(T,\gamma)=\phi(\gamma_T).
\end{array}
\right.
\end{equation}
It can also be proved~\cite{bonesini2023rough} that this solution is unique in the space where the functional Itô formula~\cite[Theorem 3.17]{viens2019martingale} is applicable.

\subsubsection{VIX options under rough Bergomi model}\label{subsubsec:VIX}
Options on VIX and realised variance form an important class of derivatives, which can be seen as options with path-dependent payoffs. We define~$\VIX_T^2 = \frac{1}{\Delta}\int_T^{T+\Delta} \bE [\psi_s(\Wh_s)\lvert \cF_t] \ds$ where~$\Delta$ is a positive constant corresponding to 30 days, and~$\Wh$ is a Gaussian Volterra process in~$\bR^e$ for any~$e\in\bN$. It is shown in~\cite{pannier2023path} that an option on the VIX (future) has the representation
$$
\bE\Big[\phi\big(\VIX_T^2\big) \big|\cF_t\Big] = \bE\bigg[\phi\circ\mathfrak{F}\Big(
    \Big\{ \Theta^t_s + J^{t,T}_s :s\in[T,T+\Delta]\Big\}\Big) \Big\lvert \Theta^t\bigg],
$$
where~$J^{t,T}_s = \Theta^T_s - \Theta^t_s$ is independent of~$\cF_t$ and $\mathfrak{F}(\gamma) :=\frac{1}{\Delta}\int_T^{T+\Delta} \bE[  \psi_s(\gamma_s+I^T_s)] \ds$ where~$I^T_s\sim\mathcal{N}\left(\bm{0}_d, \int_T^s K(s,r)K(s,r)^\top \dr\right)$. Under natural assumptions on~$\phi,\psi$ and~$K$, the map $u\big(t,\gamma): =\bE\Big[\phi\circ\mathfrak{F}\Big(
    \big\{ \gamma_s + J^{t,T}_s :s\in[T,T+\Delta]\big\}\Big)\Big]$ is the unique solution to the PPDE for all~$(t,\gamma)\in\tom$
\begin{equation}\label{eq:main_PPDE_VIX}
\def\arraystretch{1.5}
    \left\{\hspace{-.4cm}
\begin{array}{rl}
    &\displaystyle \partial_t u(t,\gamma) + \half \tpar^{K^t K^t}_{\gamma} u(t,\gamma) = 0,\\
    &\displaystyle  u(T,\gamma) = \phi\circ\mathfrak{F}(\gamma).
\end{array}
\right.
\end{equation}
\subsubsection{Options under rough Bergomi model}\label{subsec:example_rBergomi}
In this last example, which was already presented in the introduction, the underlying is the (log-)asset price in a rough Bergomi-type model.
The value function in the rough Bergomi model is indexed on the state space~$\tom$ with $d=1$. As we have already seen,
\begin{equation*}
    u(t,x,\gamma) := \bE\left[\phi(X_T) | X_{t}=x, \Theta^{t}=\gamma\right],
\end{equation*}
which entails $u(t,X_t,\Theta^t)=\bE[\phi(X_T) \lvert \cF_t]$.
Based on \cite[Theorem 2.25]{bonesini2023rough}, natural assumptions on~$\phi,\psi$ and~$K$ yield the well-posedness of the PPDE for all~$(t,x,\gamma)\in\tom$
\begin{equation}\label{eq:PPDE_rvol}
\def\arraystretch{1.5}
    \left\{\hspace{-.4cm}
    \begin{array}{rl}
    &\partial_t u+ \half\psi_t(\gamma_t) \Big(\partial_{x}^2 u-\partial_x u \Big) + \half  \tpar_{\gamma}^{K^tK^t} u + \rho\sqrt{\psi_t(\gamma_t)} \tpar_{\gamma}^{K^t} (\partial_x u) =0,\\
    &u(T,x,\gamma)=\phi(x).
\end{array}
\right.
\end{equation}
Notice the boundary condition only applies to $x$ and is not path-dependent. 

\begin{remark}
In the rough Bergomi model, the paths $\Theta$ have a financial interpretation as they are linked to the forward variance curves~$\xi$, a quantity observed on the markets. We refer to \cite[Remark 4.7 (5)]{pannier2023path} for more details.
\end{remark}

\subsubsection{An important remark on the singularity of the kernel}
We emphasise that the PPDEs presented in these examples fit in~\eqref{eq:mainPDE} albeit with a non-continuous direction~$K^t$. Since the pathwise derivative in this direction is defined as the limit of the regularised derivative with direction~$K^{t,\delta}$ as in~\eqref{eq:singular_der}, one can show thanks to the probabilistic representations that the solution to the regularised PPDE converges to the solution of the singular one as $\delta$ goes to zero. 
Our numerical scheme approximates the PPDE~\eqref{eq:mainPDE} with a regularised kernel, for a fixed~$\delta>0$.


\subsection{Functions on bounded variation paths}\label{subsec:bounded_variation}
Making the framework of subsection~\ref{subsec:functions_paths} amenable to the application of signature kernels requires some care. Firstly, the latter are defined on bounded variation paths, a subset of the continuous paths considered in the previous subsection. In this subsection, functions are thus defined over~$\Omega$ \eqref{eq:def_Omega} instead of~$\widetilde{\Omega}$ \eqref{eq:def_Omegatilde}. 

Further, we make the observation that all three examples presented above are functionals~$(t,x,\gamma)\mapsto u(t,x,\gamma)$ that, for all~$t\in\bT$, only act on~$(\gamma_s)_{s\in[t,T]}$, the path after time~$t$.
This crucial property is the reason why we want to consider paths that are constant on~$[0,t]$, so that no unnecessary information is included.
We denote by~$\cC(\Omega)$ the set of continuous functions~$u:\Omega\to\bR$ with respect to~$\bm{d}$ \eqref{eq:distance_d}; for those there exists a family~$(u_{t,x})_{t\in\bT,x\in\bR^d}$ of continuous maps from~$\BV_t=\cC^{1-var}([t,T],\bR^e)$ to~$\bR$ such that~$u(t,x,\gamma)=u_{t,x}(\gamma^0\lvert_{[t,T]},\check\gamma\lvert_{[t,T]})$ (recall that for each time-augmented path~$\gamma$ we denote by~$\gamma^0$ the strictly monotone coordinate and by~$\check{\gamma}$ the $\bR^{e-1}$-valued path without the strictly monotone coordinate). Note that the distance~$\bm{d}$ looks into the path on~$[0,t]$, indeed without this information one would lose the positivity for two paths distinct on~$[0,t]$, while considering only continuity of~$u_{t,x}$ leads to issues in defining compact sets of~$\Omega$. 
To be consistent with the derivative of the signature kernel introduced in~\cite{lemercier2021siggpde}, we choose a weaker form of pathwise derivative than the Fréchet one: the Gateaux derivative (it will appear later that signatures are continuous with respect to the supremum norm and therefore both derivatives agree). This is where the restriction to~$[t,T]$ becomes useful and allows us to avoid jumps. Recall that the Gateaux derivative perturbs the path on the whole interval where it is defined
\begin{equation}
 \partial_{\gamma}^\eta u(t,x,\gamma)
    := \lim_{\ep\downarrow0}\frac{u(t,x,\gamma+\ep\eta)-u(t,x,\gamma)}{\ep}.
\label{eq:Gateaux_def}
\end{equation}
We instead define the pathwise derivative of~$u\in\cC(\Omega)$ in the direction~$\eta\in\BV_0$ such that only the path~$\check\gamma$ is perturbed and solely on~$[t,T]$:
\begin{equation} \label{eq:Pathwise_restricted_def}
    \wpar^{\eta}_{\gamma} u(t,x,\gamma):= \partial^{({\bm 0},\check\eta\lvert_{[t,T]})}_\gamma u_{t,x}(\gamma^0\lvert_{[t,T]},\check\gamma\lvert_{[t,T]}) = \lim_{\ep\to0} \frac{u_{t,x}(\gamma^0\lvert_{[t,T]},(\check\gamma+\ep\check\eta)\lvert_{[t,T]})-u_{t,x}(\gamma^0\lvert_{[t,T]},\check\gamma\lvert_{[t,T]})}{\ep}.
\end{equation}
In contrast to the Fréchet derivative defined in~\eqref{eq:Frechet_def} that perturbs the path over the interval~$[t,T]$, this definition deals only with continuous (even bounded variation) paths. 
Similarly to the former setup in subsection~\ref{subsec:functions_paths}, for~$n,n_1,n_2,n_3\in\bN_0$, we can define the higher derivatives~$\partial_t^{n_1},\,\partial_x^{n_2}$ and~$\wpar_\gamma^{\bm{\eta}}$, for~$\bm{\eta}=(\eta_1,\cdots,\eta_{n_3})\in (\BV_0)^{  n_3}$. 
We denote~$\cC^{\cN}(\Omega)$ the 
space of functions admitting continuous such pathwise derivatives of order~$\cN\subset\bN^3_0$.
We can now define PPDEs on~$\Omega$ similar to~\eqref{eq:mainPDE}
\begin{equation}\label{eq:BV_PPDE}
\def\arraystretch{1.5}
\left\{\hspace{-.4cm}
\begin{array}{rl}
&\cL_t u(t,x,\gamma)= g(t,x,\gamma), \\ &u(T,x,\gamma)=\phi(T,x,\gamma),
\end{array}
\right.
\end{equation}
where~$g,\phi:\Omega\to\bR$ are measurable functions and the linear operator $\cL_t$ is a combination of derivatives of all arguments, analogue of~$\tcL_t$ \eqref{eq:Def_tcL} for functions independent of~$\gamma_{[0,t]}$, albeit with a pathwise derivative of Gateaux type instead of Fréchet.
More precisely, there exist~$\cN\subset\bN^3_0$, $a_n\in\bR$ for all~$n\in\cN$ and~$\bm{\eta}\in (\BV_0)^{n_3}$ such that
\begin{equation}\label{eq:def_cL_BV}
    \cL_t := \sum_{n=(n_1,n_2,n_3)\in\cN} a_n \partial_t^{n_1}\,\partial_x^{n_2} \, \wpar_\gamma^{\bm{\eta}},
\end{equation}
where the sum is taken over multi-indices~$n\in\cN$.


\subsection{A kernel method for PPDEs} \label{subsec:method}
We saw that the solution to the original PPDE of the type~$\tcL_t u = f$ on~$\tom$ is a function on continuous paths, whereas signature kernels are defined over paths of bounded variation. For this reason we defined an analogous weaker PPDE~$\cL_t u=f$ on~$\Omega\subset\tom$ (defined in~\eqref{eq:def_Omega},\eqref{eq:def_Omegatilde}) which solution can be suitably represented with signature kernels. Functions admitting Fréchet derivatives also admit Gateaux derivatives and they are equal, therefore solutions to the former PPDE are also solutions to the latter. This will allow us in Theorem~\ref{th:MainCvg} to bridge the gap between the equation we solve and the equation we target.
Now that we have this picture in mind, we can present the gist of our method, inspired by~\cite{chen2021solving} and extended to the path-dependent case. 

We consider a kernel $\kappa$, indexed on a compact subset~$\cX$ of $\Omega$, as product of an RBF kernel and a signature kernel indexed on $\mathbb R^{d+1}$ and $\BV_0=\cC^{1-var}([0,T],\bR^e)$ respectively. We call $\cH$ its associated RKHS. 
We pick~$m+n$ collocation points~$\omega^1,...,\omega^{m+n} \in \cX$ of the form~$\omega^i=(t^i, x^i, \gamma^i)$ and order them in such a way that $0 \leq t^i < T$ for $i=1,...,m$ and  $t^i = T$ for $i=m+1,...,m+n$. The former belong to the interior of the domain while the latter correspond to the boundary (terminal) condition. The main idea consists in approximating the solution of the PPDE \eqref{eq:BV_PPDE} with a minimiser of the following optimal recovery problem
\begin{align}\label{eq:OptiRecovery}
    \min_{u \in \cH} \frac{1}{2}\norm{u}^2_{\cH}  \quad \text{such that} \quad 
    \begin{cases}
        \mathcal{L}_t u(\omega^i)  = g(\omega^i) &   i = 1,...,m 
        \\
        \ \  u(\omega^i) = \phi(\omega^i) &  i  = m+1,...,m+n.
     \end{cases}
\end{align}
In other words, we search for the function with minimal RKHS norm that satisfies the PDE constraints at all the collocation points. 

This problem formulation raises a number of questions, of both theoretical and practical importance:
   \begin{enumerate}\label{enum:questions}
        \item Does the domain of $\cL_t$ contain $\cH$ ? 
        \item Does the minimisation problem \eqref{eq:OptiRecovery} have a unique solution $u_{m,n}$?
        \item How does one solve this optimisation over an infinite-dimensional space ?
        \item Does the PPDE \eqref{eq:mainPDE} have a unique classical solution $u^\star$ ? 
        \item Does $u_{m,n}$ converge to $u^\star$ as $m,n\to\infty$ ?
    \end{enumerate}
After designing our kernel and RKHS in Section~\ref{sec:kernels}, we answer positively to all these questions under appropriate conditions. In order to successively execute limit arguments, we exploit that~$\cX$ is a compact set of~$\Omega$, which we define explicitly under the topology induced by~$\bm{d}$. 
The interested reader may find a summary of the answers as follows. 
\begin{enumerate}\label{enum:answers}
    \item Yes, Proposition~\ref{proposition:C2} proves this assertion.
    \item Yes, as this is a type of optimal recovery problem. See Section \ref{sec:or-wp}.
    \item The optimal recovery problem reduces to a finite-dimensional quadratic optimisation problem with linear constraints. The latter can be solved via gradient descent, aloof from sophisticated learning algorithms. See Section \ref{sec:or-wp}.
    \item As pointed out in Section~\ref{subsec:examples}, existence and uniqueness of this PPDE holds under some conditions, which are thoroughly checked in the given references for the examples of interest. 
    \item Theorem \ref{th:MainCvg} ensures convergence towards the true solution over a compact subset $\cX$ of $\Omega$ under two different sets of assumptions, both requiring that~$u^\star$ is the unique solution of the PPDE~\eqref{eq:mainPDE} and that the collocation points form a countable dense set of $\cX$.
\end{enumerate}

\textbf{PDE Scheme.} \ \  In particular, after having checked that the constraints of the optimal recovery problem~\eqref{eq:OptiRecovery} make sense, we can solve this optimisation problem explicitly due to a variant of the representer theorem involving higher order derivatives \cite{owhadi2019operator}, which gives an answer to Question (2) of our list~\ref{enum:questions}. Notably, for all~$n,m\in\bN$ and collocation points~$(\omega^i)_{i=1}^m \in \cX^\circ$, $(\omega^{i})_{i=m+1}^{m+n}\in\partial\cX$, there exist weights~$\alpha_1, \cdots, \alpha_{m+n} \in \bR$ such that the optimal recovery problem~\eqref{eq:OptiRecovery} is well-posed and has a unique minimiser~$ u_{m,n} \in \cH$ of the form
\begin{equation}\label{eqn:exact_sol}
     u_{m,n} = \sum_{i=1}^{m+n} \alpha_i \kappa(\omega^i,\cdot).
\end{equation}
Thus, the coefficients~$\alpha_1, \cdots, \alpha_{m+n}$ can be obtained by simply solving the finite-dimensional quadratic optimisation problem with linear constraints
    \begin{equation}\label{eqn:optim_system}
        \min_{\boldsymbol \alpha \in \bR^{m+n}} \half \boldsymbol \alpha^\top \cK \boldsymbol \alpha \quad \text{subject to} \quad \widetilde \cK \boldsymbol \alpha = \boldsymbol b, 
    \end{equation}
    where $\boldsymbol \alpha = (\alpha_1,...,\alpha_{m+n}) \in \mathbb{R}^{m+n}$,  $\cK \in \mathbb{R}^{(m+n) \times (m+n)}$ is the matrix with entries $\cK_{i,j} = \kappa(\omega^i,\omega^j)$, and where $\widetilde \cK \in \mathbb{R}^{(m+n) \times (m+n)}$ and $\boldsymbol b \in \bR^{m+n}$ are defined as follows
\begin{equation}\label{eqn:linear_system}
    \widetilde \cK := 
    \begin{pmatrix}
        \mathcal{L}_t \kappa(\omega^1,\omega^1) & ... & \mathcal{L}_t \kappa(\omega^1,\omega^{m+n}) \\
        \vdots & \ddots & \vdots 
        \\
        \mathcal{L}_t \kappa(\omega^m,\omega^1) & ... & \mathcal{L}_t \kappa(\omega^m,\omega^{m+n})\\
        \kappa(\omega^{m+1},\omega^1) & ... &  \kappa(\omega^{m+1},\omega^{m+n})
        \\
        \vdots & \ddots & \vdots 
        \\
        \kappa(\omega^{m+n},\omega^1) & ... &  \kappa(\omega^{m+n},\omega^{m+n})
    \end{pmatrix}
    \quad 
    \text{and}
    \quad 
    \boldsymbol b := 
    \begin{pmatrix}
        g(\omega^1) \\
        \vdots\\
        g(\omega^m) \\
        \phi(\omega^{m+1})\\
        \vdots \\
        \phi(\omega^{m+n})
    \end{pmatrix}.
\end{equation}

This lifts the main computational obstacle and opens a clear path to solving the optimisation problem, thereby resolving Question (3) of~\ref{enum:questions}.

\begin{remark}\label{remark:2}
For clarity, we are not able at this stage to verify the assumptions of Theorem \ref{th:MainCvg} in our examples of interest. We believe however that this paper is the first step towards a more general theory, which shall grow as our understanding of these PPDEs develop. In particular,  the assumption that $\mathcal{X}$ is a compact subspace of $\Omega$ (the space of bounded variation paths) could be lifted using at least two strategies we can think of: 1) by carefully rescaling the signature coefficients in a path-by-path exploiting the concept of \emph{robust signatures} proposed in~\cite{chevyrev2018signature}; or 2) considering a different class of topologies on \emph{weighted-spaces} as discussed in~\cite{cuchiero2023global}. Both of these approaches essentially ensure that the range of the signature map remains "bounded" even when the domain is non-compact. The former does so by changing the feature map, whilst the second changes the underlying path-space topology. At the cost of increased technicality one can also replace the space of bounded variation with a space of arbitrary $p$-variation.  Uniqueness of the solution to the PPDE should thus be attainable by classical probabilistic techniques (i.e. Feynman--Kac theorem). For more details about the assumptions we refer to the discussion below Theorem \ref{th:MainCvg}.
\end{remark}

\begin{remark}
    We stress that kernel-based supervised learning approaches like kernel ridge regression require a dataset of input–output pairs, i.e. collocation paths together with corresponding solution values. In our setting such dataset is unavailable: the solution of the PPDE is exactly what we seek to compute. Our solver is unsupervised, in the sense that it enforces the PPDE operator at collocation points as constraint in the kernel-based optimisation in~\eqref{eq:OptiRecovery}, without needing prior solution values. Moreover, while conditional expectations can represent PPDE solutions in special cases, for general PPDEs no such probabilistic representation is guaranteed. Our approach therefore applies in broader settings where regression-based methods are not directly applicable. 

\end{remark}

\begin{remark}
    We note that an alternative approach to our optimal recovery problem would consist in considering instead a ``kernel ridge regression''--type formulation, in which one minimises a penalised empirical loss of the PDE residual over the RKHS under a norm constraint or regularisation. A prototypical example is
    \begin{equation*}
        \min_{u \in \cH} 
        \sum_{i=1}^m \big|\mathcal{L}_t u(\omega^i) - g(\omega^i)\big|^2 
        + \sum_{i=m+1}^{m+n} \big|u(\omega^i) - \phi(\omega^i)\big|^2
        + \lambda \|u\|_{\cH}^2 ,
    \end{equation*}
    for some regularisation parameter $\lambda > 0$, or more generally a formulation based on an $L^2$--loss of the PDE residual over a measure on the domain. Results of this type can be found, for instance, in \cite[Theorem~2]{vito2004some}. Such kernel ridge regression (KRR) discretisation based on the \emph{same} kernel ansatz
    \[
    u(\cdot) \;=\; \sum_{j=1}^{m+n} \alpha_j \,\kappa(\cdot,\omega^j)
    \]
    leads to the unconstrained penalised least-squares problem
    \begin{equation*}
        \min_{\boldsymbol \alpha \in \bR^{m+n}} 
        \;\big\| \widetilde \cK\,\boldsymbol \alpha - \boldsymbol b \big\|_2^2
        \;+\; \lambda \,\boldsymbol \alpha^\top \cK \boldsymbol \alpha,
        \qquad \lambda > 0,
    \end{equation*}
    where \emph{the same matrices} $\cK$, $\widetilde \cK$ and right-hand side $\boldsymbol b$ appear.  The first term measures the squared $\ell^2$-residual of the PDE and boundary/terminal constraints at the collocation points, while the second term penalises the RKHS norm of $u$ via
    \[
    \|u\|_{\cH}^2 = \boldsymbol \alpha^\top \cK \boldsymbol \alpha .
    \]
    The associated first-order optimality conditions yield the normal equations
    \begin{equation*}
        \big(\widetilde \cK^\top \widetilde \cK + \lambda \cK \big)\boldsymbol \alpha
        \;=\;
        \widetilde \cK^\top \boldsymbol b .
    \end{equation*}

Thus, both discretisations are built from the same kernel expansion and the same matrices $(\cK,\widetilde \cK,\boldsymbol b)$.  
The only structural difference is that the optimal recovery formulation imposes
\[
\widetilde \cK\,\boldsymbol \alpha = \boldsymbol b
\]
as hard constraints and minimises the RKHS norm subject to these constraints, whereas kernel ridge regression replaces these constraints by a quadratic penalty and enforces them only in a least-squares sense. In particular, for any fixed $\lambda>0$ one generally has
\[
\widetilde \cK\,\boldsymbol \alpha \neq \boldsymbol b,
\]
so that the PDE and boundary/terminal conditions are not satisfied exactly at the collocation points, with the trade-off between fidelity to the constraints and RKHS norm controlled by the regularisation parameter $\lambda$. Formally, the optimal recovery problem may be viewed as the zero-regularisation or infinite-penalty limit of kernel ridge regression: as $\lambda \downarrow 0$, any accumulation point of minimisers of the penalised problem that satisfies 
$\widetilde \cK\,\boldsymbol \alpha = \boldsymbol b$ coincides with the minimum-norm solution of the constrained problem. Conversely, for fixed $\lambda>0$, the kernel ridge solution does not in general satisfy the PDE exactly at the collocation points.
\end{remark}

Finally, let us summarise the several levels of approximation. We numerically solve the optimisation problem~\eqref{eq:OptiRecovery}, which is a discretisation of the PPDE~\eqref{eq:BV_PPDE} over~$\Omega$. Under the right set of assumptions, the sequence of solutions to the optimisation problem converges the solution to the PPDE~\eqref{eq:mainPDE} over~$\widetilde{\Omega}$. Finally, the latter is an approximation of the PPDE with a singular direction~$\bm{\eta}$. While our setup is designed to match PPDEs of second generation (arising from Volterra processes), it could be adapted to the first generation of Dupire (corresponding to path-dependent payoffs) at the cost of redeveloping the theory with the corresponding spaces, notions of derivatives and topologies.

\section{Signature kernels}
\label{sec:kernels}



Our numerical solver is based on so-called signature kernels, a special class of kernels indexed on $\BV_0=\cC^{1-var}([0,T],\bR^e)$. In this section we enumerate the main properties of such signature kernels and highlight important aspects related to their numerical evaluation. We begin by recalling definition and properties of classical kernels.

\subsection{Classical kernels} A scalar-valued kernel on a set $\widehat{\Omega}$ is a symmetric function of the form $\kappa : \widehat{\Omega} \times \widehat{\Omega} \to \mathbb{R}$. Such kernel is said to be positive semidefinite if for any $n \in \mathbb{N}$ and points $\omega^1,...,\omega^n \in \widehat{\Omega}$, the Gram matrix $\mathcal{K}:=(\kappa(\omega^i,\omega^j))_{i,j}$ is positive semidefinite, i.e. if for any $\alpha_1,...,\alpha_n \in \mathbb{R}$
\begin{equation*}
    \sum_{i=1}^n\sum_{j=1}^n \alpha_i \alpha_j\kappa(\omega^i,\omega^j) \geq 0.
\end{equation*}
Recall that a Hilbert space $\cH$ of functions defined on a set $\widehat{\Omega}$ is a \emph{reproducing kernel Hilbert space} (RKHS) over $\widehat{\Omega}$ if, for each $\omega \in \widehat{\Omega}$, the point evaluation functional at $\omega$, $f \mapsto f(\omega)$,  is a continuous linear functional, i.e. there exists a constant $C_\omega \geq 0$ such that
$$|f(\omega)| \leq C_\omega \norm{f}_\cH, \quad \text{for all } f \in \cH.$$
The classical Moore-Aronszajn Theorem \cite{aronszajn1950theory} ensures that if $\kappa$ is a positive semidefinite kernel, then there exists a unique RKHS $\cH$ such that $\kappa$ has the following \emph{reproducing property}  
\begin{equation}\label{eqn:reproducing_property}
    \left\langle \kappa(\omega,\cdot), f\right\rangle_\cH = f(\omega), \quad \text{for all } f \in \cH \text{ and } \omega \in \widehat{\Omega}.
\end{equation}
When dealing with questions related to function approximation, one is often interested in understanding how well elements of an RKHS can approximate a function in a given class. One important such class is the family of continuous functions, leading to the notion of \emph{cc-universality}  \cite{sriperumbudur2010hilbert}. If the set~$\widehat{\Omega}$ is equipped with an arbitrary topology and $\kappa:\widehat{\Omega}\times\widehat{\Omega}\rightarrow\mathbb{R}$ is a continuous,
symmetric, positive semidefinite kernel on $\widehat{\Omega}$ with RKHS $\cH$, then we say that $\kappa$ is
cc-universal on $\widehat{\Omega}$ if, for every compact subset $\mathcal{X}\subset\widehat{\Omega}$, $\cH$ is dense in $\cC(\mathcal{X})$ in the topology of uniform convergence.

Kernels on Euclidean spaces such as linear, polynomial, Gaussian and Mat\'ern kernels have been extensively studied in prior literature. Arguably, the most popular choice (which is also the one we make in our experiments) of cc-universal kernel on $\bR_+\times\bR^d$ is the \emph{radial basis function} (RBF) kernels $g_\sigma : \bR_+\times\bR^d \times \bR_+\times\bR^d\to \bR$ defined for $(s,x),(t,y) \in \bR_+\times\bR^d$ as
\begin{equation}\label{eqn:rbf_kernel}
    g_\sigma((s,x),(t,y)) = \exp \left( -\frac{\norm{(s,x)-(t,y)}^2}{2\sigma^2}\right), \quad  \sigma \in \bR.
\end{equation}
To handle solutions of PPDEs we need to introduce kernels indexed on $\BV_0$. \emph{Signature kernels}  \cite{kiraly2019kernels, salvi2021signature} provide a natural class of cc-universal of kernels on $\BV_0$ as we shall discuss next. 

\subsection{Signature kernels} 
The definition of signature kernel requires an initial algebraic setup.  Denote by $\otimes$ the standard tensor product of vector spaces. Let $V$ be a Hilbert space with inner product~$\langle \cdot, \cdot \rangle_V$. In this section, we adopt a more general notation, although we will specifically consider $V$ as the vector space $\mathbb{R}^e$ in the sequel of the paper. For any $n \in \bN$, the canonical Hilbert-Schmidt inner product $\langle \cdot, \cdot \rangle_{V^{\otimes n}}$ on $V^{\otimes n}$ is defined as
\begin{equation}\label{eqn:inner_prod_n}
    \left\langle a,b \right\rangle_{V^{\otimes n}} = \prod_{i=1}^n \left\langle a_i,b_i \right\rangle_V
\end{equation}
where $a=(a_1,...,a_n)$ and $b=(b_1,...,b_n)$ are tensors in $V^{\otimes n}$. Let $T((V)) = \prod_{n=1}^\infty V^{\otimes n}$ the vector space of formal tensor series and $T(V) \subset T((V))$ the space of formal tensor polynomials defined as 
\begin{equation*}
    T(V) = \{a \in T((V)) \mid a \text{ has a finite number of non-zero entries}\}.
\end{equation*}
There are many ways of extending by linearity the inner product in~\eqref{eqn:inner_prod_n} to an inner product on $T(V)$. The simplest way to achieve this is by defining the inner product as follows
\begin{equation}\label{eqn:l_ip}
    \langle a, b \rangle_{T(V)} = \sum_{n=0}^\infty \left\langle a_n, b_n \right\rangle_{V^{\otimes n}}
\end{equation}
for any $a=(a_0, a_1,...), b=(b_0,b_1...)$ in $T(V)$. 

\begin{remark}
    A more general class of inner products was introduced in \cite{cass2024weighted} by means of a weight function  $\phi : \bN_0\to \bR_+$ in the following way
\begin{equation}\label{eqn:w_ip}
    \langle a, b \rangle_{T(V)} = \sum_{n=0}^\infty \phi(n) \left\langle a_n, b_n \right\rangle_{V^{\otimes n}}.
\end{equation}
\end{remark}
We denote by $\overline{T(V)}$ the Hilbert space obtained by completing $T(V)$ with respect to $\langle \cdot, \cdot \rangle_{T((V))}$. 
Supposing that $V$ is $d$-dimensional, for some $d > 1$, it will be convenient to introduce the set of \emph{words}~$\mathbb{W}$ in the non-commuting letters $\{1,\dots, d\}$, i.e. ordered $m$-tuples $w = (i_1,i_2,\cdots, i_m) \in \{1,\dots, d\}^m$ with $|w| := m \in \mathbb N$.
The principal ingredient needed to define the signature kernel is a classical transform in stochastic analysis known as the \emph{signature}. Next we recall its definition. 
Let $\BV(\bT,V)$ be defined as in Section \ref{subsec:path_spaces}
where the norm in the definition of $p$-variation is the norm in $V$. For any closed interval $[t, t'] \subset \bT$ and any path $\gamma\in \BV(\bT,V)$, the signature of $\gamma$ over $[t, t']$ is the following infinite collection
$$S(\gamma)_{t,t'} = \left(S(\gamma)^{(w)}_{t,t'}\right)_{w \in \mathbb{W}} \in \overline{T(V)} \cong l^2(\mathbb W)$$ 
 of Riemann-Stieltjes iterated integrals
\begin{equation}\label{eqn:iterated_integrals}
S(\gamma)^{(w)}_{t,t'} = \int_{t<t_{1}<...<t_{m}<t'} \D\gamma_{t_{1}}^{(i_1)} ... \D\gamma_{t_m}^{(i_m)}, \quad \text{for any } w=(i_1,...,i_m) \in \mathbb W.
\end{equation}

When it is clear from the context, we will suppress the dependence on the interval and instead denote the signature of $\gamma$ by $S(\gamma)$. We refer the interested reader to \cite{kidger2019deep, fermanian2023new} for an account and examples on the use of signature methods in machine learning.

The signature of a BV path $\gamma$ can be equivalently defined as the solution of a $\overline{T(V)}$-valued differential equations controlled by $\gamma$, as stated in the following lemma.

\begin{lemma}\label{lemma:sig_cde}
    Let $\gamma \in \BV(\bT,V)$. Then, the linear controlled differential equation 
    \begin{equation}\label{cde sig}
    dY_s=Y_s \cdot \D\gamma_s  \text{ on $[t,t']$, \ started at } A \in \overline{T(V)} 
    \end{equation}
    admits $Y_s = A \cdot S(\gamma)_{[t,s]}$ as unique solution, where the product $\cdot$ on $\overline{T(V)}$ is defined for any \newline $v=\left(v_{0},v_{1},...\right)  $ and $w=\left(  w_{0},w_{1},...\right)$ in
    $\overline{T(V)}$ as the element
    $v \cdot w=\left( z_{0},z_{1},...\right)$  in
    $\overline{T(V)}$  such that
    \begin{equation*}
    z_k = \sum_{j=0}^k v_j \otimes w_{k-j} \in V^{\otimes k}, \quad \forall k \in \mathbb{N}
    \cup\left\{  0\right\},
    \end{equation*}
    and where $\D\gamma_t$ is identified with $(0,\D\gamma_t,0,...) \in \overline{T(V)}$. In particular, the signature $S(\gamma)_{t,t'}$ is the unique solution to (\ref{cde sig}) with initial condition $A = \mathbf{1} := (1,0,0,...) \in \overline{T(V)}$.
\end{lemma}

    
\begin{proof}
    For any $a \in V$, the map $f_a:\overline{T(V)} \rightarrow L\left(  V,\overline{T(V)}\right)$ given by
    $f_a \left(  Y\right) = Y\cdot a$ is well defined and linear. Thus, it can
    easily been seen to satisfy the conditions of the classical versions of Picard-Lindel\"of theorem for Young CDEs (e.g. \cite[Thm. 1.28]{lyons2007differential}). Therefore, (\ref{cde sig}) admits a unique solution in $\overline{T(V)}$. 
    The fact that any solution $Y_s=\left(  Y_s^{0},Y_s^{1},Y_s
    ^{2},...\right)$ to (\ref{cde sig}) in $\overline{T(V)}$ must satisfy $Y_s^{k}= (A \cdot S\left(  \gamma\right)_{t,s})^{k}$ for any $k$ in $\mathbb{N}$ and any $s$ in $\left[t,t'\right]$ can be shown by a simple induction on $k$.
\end{proof}

\begin{remark}
    By construction, the iterated integrals in (\ref{eqn:iterated_integrals}) admit the following recursive structure
    \begin{equation*}
        S(\gamma)^{(w)}_{t,t'} = \int_{t}^{t'}S(\gamma)^{(w_-)}_{t,t'} d \gamma^{(i_m)}_u, \quad \text{where } w_- = (i_1,...,i_{m-1}).
    \end{equation*}
    Thus the coordinates of the signature, when taken in isolation, do not satisfy a differential equation.
\end{remark}

We can now define the signature kernel as the inner product of two signatures.
\begin{definition}\label{def:signature_kernel}
    Let~$0\le s<s'\le T,\,0\le t<t'\le T$. Define the signature kernel $\kappa_{\text{sig}} :\BV([t,t'],V)\times\BV([s,s'],V)\to \bR$~as the following symmetric function 
    \begin{equation*}
        \kappa_{\text{sig}}(\gamma,\tau) = \left\langle S(\gamma)_{t,t'}, S(\tau)_{s,s'} \right\rangle_{\overline{T(V)}} = \sum_{k=0}^\infty \sum_{w \in \mathbb W : |w| = k} S(\gamma)_{t,t'}^{(w)}S(\tau)_{s,s'}^{(w)}.
    \end{equation*}
\end{definition}
\begin{remark}
    Note that the two paths $\gamma,\tau$ need not be defined on the same time interval. As for the norms, we will omit the subscript~$t,t'$ when it is clear from the context.
\end{remark}

It was proved in \cite{salvi2021signature} that the signature kernel can be realised as the solution of a Goursat PDE as stated in the next lemma; this effectively provides a \emph{kernel trick} for the signature kernel, i.e. a way of computing it without reference to the signature map. 
\begin{lemma}\cite[Theorem 2.5]{salvi2021signature}\label{lemma:signature_kernel_trick} Let~$[s,s'],[t,t'] \in [0,T]$ and $\gamma\in\BV([t,t'],V),\,\tau \in \BV([s,s'],V)$ be two paths. Then, the signature kernel of $\gamma, \tau$ is realised as the solution of the following two-parameter integral equation 
\begin{equation}\label{eqn:signature_kernel_PDE}
    \kappa_{\text{sig}}(\gamma, \tau) = \kappa_{\text{sig}}(\gamma, \tau)_{t',s'} \quad \text{where} \quad \kappa_{\text{sig}}(\gamma, \tau)_{r,q} = 1 + \int_s^{q}\int_t^{r} \kappa_{\text{sig}}(\gamma, \tau)_{u,v}\langle \D\gamma_u,d\tau_v \rangle_V,
\end{equation}
and
with boundary conditions $\kappa_{\text{sig}}(\gamma, \tau)_{t,q} = \kappa_{\text{sig}}(\gamma, \tau)_{r,s} = 1$ for all~$q\in[s,s'],r\in[t,t']$.
\end{lemma}

Several finite difference schemes are available for numerically evaluating solutions to equation~\eqref{eqn:signature_kernel_PDE}, see \cite[Section 3.1]{salvi2021signature} for details. 

\subsection{Signature kernels from static kernels}
In view of data science applications, instead of taking inner products of signatures directly, it might be beneficial to first lift paths in $V$ to paths with values in the RKHS $\mathcal{H}_g$ of a static kernel $g : V \times V \to \bR$. More precisely, for any $\gamma \in \BV([t,t'],V)$, with~$0\le t<t'\le T$, denote by $\gamma^g : [t,t'] \to \mathcal{H}_g$ the lifted path defined for any $r\in [t,t']$ as $\gamma^g_r = g(\gamma_r,\cdot)$. For the signature of a lifted path $\gamma^g$ to be well-defined we will require that the latter is at least continuous and of bounded variation, i.e. $\gamma^g \in \BV([t,t'],\mathcal{H}_g)$. The following result gives a sufficient condition to ensure this and immediately follows from the definitions of bounded variation and Lipschitz continuity. In particular, in the experimental section, we will make use of the \emph{radial basis function} kernel as $g$, which is $\cC^{\infty}$ and Lipschitz continuous.
\begin{lemma}
    Let $g : V \times V \to \bR$ be a kernel such that for any $v \in V$ the canonical feature map $g(v,\cdot)$ is Lipschitz-continuous and denote by $\cH_g$ its RKHS. Then, for any path $\gamma \in \BV([t,t'],V)$, the lifted path $\gamma^g : [t,t'] \to \mathcal{H}_g$ defined for any $r \in [t,t']$ as $\gamma^g_r= g(\gamma_r, \cdot)$ is continuous and of bounded variation.
\end{lemma}
Once paths in $\BV([t,t'],V)$ are lifted to paths in $\BV([t,t'],\mathcal{H}_g)$ one can compute their signatures and take their inner products in the same way as before. More precisely, given any path $\gamma \in \BV([t,t'],V)$, the signature $S(\gamma^g)$ of the lifted path $\gamma^g \in \BV([t,t'],\cH_g)$ is a well-defined element of the closure of $T(\mathcal{H}_g)$.
This yields the following definition of lifted signature kernel.

\begin{definition}\label{def:kappag}
    Let $g : V \times V \to \bR$ be a kernel such that for any $v \in V$ the canonical feature map $g(v,\cdot)$ is Lipschitz-continuous and denote by $\cH_g$ its RKHS. Then, for any $[s,s'], [t,t'] \subset [0,T]$, the $g$-lifted signature kernel $\kappa^g_{\text{sig}} : \BV([t,t'],V) \times \BV([s,s'],V) \to \bR$ is defined as follows
    \begin{equation*}
        \kappa^g_{\text{sig}}(\gamma,\tau) = \left\langle S(\gamma^g)_{t,t'},S(\tau^g)_{s,s'}\right\rangle
    \end{equation*} 
    where the inner product is taken in the closure of $T(\mathcal{H}_g)$.
\end{definition}
In the sequel, the intervals~$[t,t']$ and~$[s,s']$  will be understood implicitly from the domain of definition of~$\gamma$ and~$\tau$.  
\Cref{lemma:signature_kernel_trick} directly yields a kernel trick for the $g$-lifted signature kernel $\kappa^g_{\text{sig}}$: 
\begin{equation}\label{eqn:lifted_kernel}
    \kappa^g_{\text{sig}}(\gamma, \tau) = \kappa^g_{\text{sig}}(\gamma, \tau)_{t',s'} \quad \text{where} \quad \kappa^g_{\text{sig}}(\gamma, \tau)_{r,q} = 1 + \int_s^q\int_t^r \kappa^g_{\text{sig}}(\gamma, \tau)_{u,v}\left\langle \D\gamma^g_u,d\tau^g_v \right\rangle_{\cH_g},
\end{equation}
with the boundary conditions
$\kappa^g_{\text{sig}}(\gamma, \tau)_{t,q} = \kappa^g_{\text{sig}}(\gamma, \tau)_{r,s} = 1$ for all~$q\in[s,s'],r\in[t,t']$.

\begin{remark}
    The integration with respect to $\cH_g$ paths is understood in the sense of
    \begin{equation}\label{eq:chain_rule_dg}
    \D\gamma^g_u = \D g(\gamma_u,\cdot)= \sum_{i=1}^d \partial_{x^i} g(\gamma_u,\cdot) \D \gamma^i_u.
    \end{equation}
\end{remark}

\subsection{Derivatives of static kernels}
To approximate solutions to PPDEs, it will be necessary to differentiate a kernel with respect to its input variables. The RBF kernel $g_\sigma$, defined in \eqref{eqn:rbf_kernel}, is clearly~$\cC^\infty$ and its $n^{th}$ order derivative can be obtained using elementary calculus. For example, the first two derivatives admit the following expressions
\begin{align}\label{eq:Gaussian_Kernel_der}
    \frac{\partial g_\sigma(x,y)}{\partial x} &=  \frac{y-x}{\sigma^2}\exp\left(-{\frac{(x-y)^2}{2 \sigma^2}}\right) \\
    \frac{\partial^2 g_\sigma(x,y)}{\partial x^2} &=  \frac{(x-y)^2-\sigma^2}{\sigma^4}\exp\left(-{\frac{(x-y)^2}{2 \sigma^2}}\right)
\end{align}

By the results in \cite{zhou2008derivative}, for all $x\in V$, the derivatives of $g(x,\cdot)$ are elements of $\cH_g$. Furthermore, \cite[Theorem 1]{zhou2008derivative} states that, for any $h\in \cH_g$ and any multi-index $I$,
 $$
 \langle \partial_{x^I}^{\abs{I}} g(x,\cdot), h \rangle_{\cH_g} = \partial_{x^I}^{\abs{I}}  h(x) = \partial_{x^I}^{\abs{I}}  \langle  g(x,\cdot), h \rangle_{\cH_g},
 $$
 where $x^I := \{x^i, i\in I\}$.
 For $n,m\in\bN$ let $I,J$ be multi-indices on $\{1,\cdots,n\}$ and $\{1,\cdots,m\}$ respectively. 
In particular, we have
\begin{equation}\label{eq:reproducing_dg}
\langle \partial_{x^I}^n g(x,\cdot), \partial_{y^J}^{m}  g(y,\cdot) \rangle_{\cH_g}  
= \partial_{x^I}^n \partial_{y^J}^{m}  \langle g(x,\cdot),  g(y,\cdot) \rangle_{\cH_g} 
=  \partial_{x^I}^n  \partial_{y^J}^{m}  g(x,y).
\end{equation}
We make the following standing assumption on the static kernel.
\begin{assumption}\label{assu:statickernel}
$g:V\times V\to\bR$ is four times differentiable and, for all $0\le n,m\le 4$, we have
    \begin{align*}
    \overline{g_{nm}}:=\sup_{x,y\in V} \left(\sum_{\abs{I}=n,\, \abs{J}=m}\abs{\partial_{x^I}^n \,\partial_{y^J}^{m} \, g(x,y)}^2 \right)^\half <\infty.
\end{align*}
\end{assumption}
This assumption is satisfied for the RBF kernel \eqref{eqn:rbf_kernel}.

\subsection{Derivatives of signature kernels} In this section we will only consider paths on intervals of the form~$[t,T],[s,T]$, for~$s,t\in\bT$ and fix~$V=\bR^e$; this way the path spaces boil down to those introduced in Section~\ref{subsec:path_spaces}.
For two paths~$\gamma\in\BV_t$ and~$\tau\in\BV_s$, 
the directional derivative of the signature kernel with respect to its first variable and along the direction of a path $\eta\in\BV_t$ is the Gateaux derivative~\eqref{eq:Gateaux_def} which perturbs the whole path:
\begin{equation}\label{eqn:dir_derivative}
    \partial_\gamma^\eta \kappa_{\text{sig}}(\gamma,\tau) := \lim_{\ep \to 0}\frac{\kappa_{\text{sig}}(\gamma + \ep \eta,\tau) - \kappa_{\text{sig}}(\gamma,\tau)}{\ep}
\end{equation}
It was shown in \cite{lemercier2021siggpde} that the directional derivative in equation~\eqref{eqn:dir_derivative} is such that
\begin{equation*}
    \partial_\gamma^\eta \kappa_{\text{sig}}(\gamma,\tau) = \partial_\gamma^\eta\kappa_{\text{sig}}(\gamma,\tau)_{T,T}
\end{equation*}
where the map $\partial_{\gamma}^{\eta}\kappa_{\text{sig}}(\gamma,\tau) : [t,T] \times [s,T] \to \bR$ solves a two-parameter integral equation  similar to~\eqref{eqn:signature_kernel_PDE}
\begin{align}\label{eqn:pde_derivative}
    \partial_\gamma^\eta \kappa_{\text{sig}}(\gamma,\tau)_{s',t'} = \int_s^{s'}\int_t^{t'} \Big(\partial_\gamma^\eta\kappa_{\text{sig}}(\gamma,\tau)_{u,v}  \big\langle \D \gamma_u,\D\tau_v\big\rangle
    +\kappa_{\text{sig}}(\gamma,\tau)_{u,v} \big\langle \D\eta_u,\D\tau_v\big\rangle \Big).
\end{align}
Leveraging  \eqref{eqn:lifted_kernel}, we can extend this PDE formulation to the lifted kernel and to its second derivative, as we show in Proposition \ref{prop:der_kappag}. For all $\gamma\in \BV_t$, define the map~$G : \BV_t \to \BV([t,T],\cH_g)$  as 
$$
G(\gamma)_u = \gamma^g_u = g(\gamma_u,\cdot), \quad \text{for all   }u\in[t,T].
$$
The proofs of all the results of this section are postponed to Appendix \ref{sec:proofs_kernel} to ease the flow of the paper.
\begin{proposition}\label{prop:der_kappag}
    Let Assumption~\ref{assu:statickernel} holds and~$\gamma \in \BV_t,\tau\in\BV_s$ for~$s,t \in \bT$. Then the following hold for all~$s'\in[s,T]$ and~$t'\in[t,T]$.
    \begin{enumerate}
        \item 
    For any $\eta\in \BV_t$, the directional derivative of the lifted kernel satisfies the following PDE
\begin{align}\label{eqn:pde_derivative_G}
     \partial_\gamma^\eta \kappa^g_{\text{sig}}(\gamma,\tau)_{s',t'} 
     =\int_s^{s'}\int_t^{t'} \left(\partial_\gamma^\eta\kappa^g_{\text{sig}}(\gamma,\tau)_{u,v} \left\langle \D \gamma^g_u, \D \tau^g_v \right\rangle_{\cH_g} + \kappa^g_{\text{sig}}(\gamma,\tau)_{u,v} \left\langle\D\partial^\eta_\gamma G(\gamma)_u, \D \tau^g_v \right\rangle_{\cH_g} \right).
\end{align}
\item For any $\eta,\etab\in \BV_t$, the second derivative of the lifted kernel satisfies the following PDE
    \begin{align}
        \partial_\gamma^{\eta\etab} \kappa^g_{\text{sig}}(\gamma,\tau)_{s',t'} =& \int_s^{s'} \int_t^{t'}  \partial_\gamma^{\eta\etab} \kappa^g_{\text{sig}}(\gamma,\tau)_{u,v} \langle \D\gamma^g_u, \D\tau^g_v\rangle_{\cH_g}
    +\int_s^{s'} \int_t^{t'}  \partial_\gamma^{\eta} \kappa^g_{\text{sig}}(\gamma,\tau)_{u,v} \langle \D\partial_\gamma^{\etab} G(\gamma)_u, \D\tau^g_v\rangle_{\cH_g} \\
    &+ \int_s^{s'} \int_t^{t'}  \partial_\gamma^{\etab} \kappa^g_{\text{sig}}(\gamma,\tau)_{u,v} \langle \D\partial^\eta_\gamma G(\gamma)_u, \D\tau^g_v\rangle_{\cH_g}
    + \int_s^{s'} \int_t^{t'}  \kappa^g_{\text{sig}}(\gamma,\tau)_{u,v} \langle \D\partial^{\eta\etab}_\gamma G(\gamma)_u, \D\tau^g_v\rangle_{\cH_g}.\nonumber
    \end{align}
\end{enumerate}
\end{proposition}
From the PDEs of~\Cref{prop:der_kappag} arise uniform estimates in~$s,t\in\bT$ for the signature kernel and its derivatives with respect to the $1$-variation norms of~$\gamma,\tau,\eta,\etab$. The precise statements and bounds are gathered in Lemma~\ref{lemma:bound_kappag_appendix} in the Appendix. 
In Section~\ref{sec:theory} we will restrict the state space to a compact subspace of~$\BV_0$ with bounded $1$-variation norm which yield uniform bounds.
\begin{remark}
Expliciting the partial derivatives in \eqref{eqn:pde_derivative_G} yields the equivalent formulation:
\begin{align}\label{eqn:pde_derivative_partial}
    \partial_\gamma^\eta \kappa^g_{\text{sig}}(\gamma,\tau)_{s',t'} &= \int_s^{s'} \int_t^{t'}  \partial_\gamma^\eta\kappa^g_{\text{sig}}(\gamma,\tau)_{u,v} \sum_{i,k=1}^e \partial_{x^i y^k}^2 g(\gamma_u,\tau_v) \D \gamma^i_u \D \tau^k_v  \\
    &\quad + \int_s^{s'} \int_t^{t'} \kappa^g_{\text{sig}}(\gamma,\tau)_{u,v} \sum_{i,j,k=1}^e \partial_{x^{ij} y^k}^3
    g(\gamma_u,\tau_v) \eta_u^j \D \gamma_u^i\D\tau^k_v  \nonumber\\
    &\quad + \int_s^{s'} \int_t^{t'} \kappa^g_{\text{sig}}(\gamma,\tau)_{u,v} \sum_{i,k=1}^e \partial_{x^i y^k}^2 g(\gamma_u,\tau_v)\D \eta_u^i  \D \tau^k_v.\nonumber
\end{align}    
Equation \eqref{eqn:pde_derivative_G} is more conducive for numerical computations while Equation \eqref{eqn:pde_derivative_partial} will be used for the theoretical analysis. One can perform the same expansion for the second derivative, as can be seen in the proof (Equation \eqref{eqn:pde_derivative_partial2}). 
\end{remark}
\begin{remark}\label{rem:kappa_C2}
    By iterating the same arguments one could prove that the pathwise derivative of the kernel at any order satisfies a linear PDE of the same kind. We do not pursue this here, but note in particular that $\kappa(\cdot,\tau)$ is infinitely many times differentiable for all $\tau\in \BV(\bT,\bR^e)$.
\end{remark}

Similarly to the kernel case, given two paths $\gamma,\eta \in \BV_t$, we can define the directional derivative of the signature on lifted paths
\begin{equation}\label{sig_dir_der}
    \partial_\gamma^\eta S\circ G(\gamma)_{t,t'} = \partial_\gamma^\eta S(\gamma^g)_{t,t'} := \lim_{\ep \to 0} \frac{S((\gamma + \ep \eta)^g)_{t,t'} - S(\gamma^g)_{t,t'}}{\ep},
\end{equation}
as well as the second derivative with respect to another direction~$\etab\in\BV(\bT,V)$, denoted~$\partial^{\eta\etab}_\gamma S\circ G$. The following result is a corollary of Lemma~\ref{lemma:bound_sig_appendix}, which also provides estimates of the signature and its derivatives with respect to the $1$-variation norms of~$\gamma,\eta,\etab$.
\begin{proposition}\label{prop:cty_signature}
    For all~$\gamma,\eta,\etab\in\BV(\bT,\bR^e)$, $e\in\bN$, the signature and its first two directional derivatives $\partial^\eta_\gamma S\circ G$ and~$\partial^{\eta\etab}_\gamma S\circ G$  exist and are continuous with respect to the supremum norm.  Moreover, the derivatives satisfy the CDEs~\eqref{eq:PDE_dersig} and~\eqref{eq:PDE_dertwosig}.
\end{proposition}
This result is analogous to \cite[Theorem 4.4]{friz2010multidimensional} albeit in $\cH_g$-valued paths.


\subsection{Computations of kernel derivatives}

In practice, signature kernels and their  derivatives can be computed with a single call to a PDE solver. To see this, for any $\gamma, \tau, \eta, \etab \in \BV_0$ and any $s,t \in \bT$ set
\begin{align*}
    \mathbb A_{s,t}(\gamma,\tau) = \left\langle\partial_t\gamma^g_s, \partial_t \tau^g_t \right\rangle_{\cH_g},  \quad
    \mathbb A^\eta_{s,t}(\gamma,\tau) = \left\langle\partial_t\partial^\eta_\gamma G(\gamma)_s, \partial_t \tau^g_t \right\rangle_{\cH_g}, \quad
    \mathbb A^{\eta \etab}_{s,t}(\gamma,\tau) = \left\langle\partial_t\partial^{\eta\etab}_\gamma G(\gamma)_s, \partial_t \tau^g_t \right\rangle_{\cH_g}
\end{align*}
Then, the augmented variable
$$\mathbf K_{s,t}^{\eta\etab}(\gamma,\tau) = \begin{pmatrix}
    \kappa^g_{\text{sig}}(\gamma,\tau)_{s,t} \\
    \partial^\eta_\gamma\kappa^g_{\text{sig}}(\gamma,\tau)_{s,t} \\
    \partial^\etab_\gamma\kappa^g_{\text{sig}}(\gamma,\tau)_{s,t}\\
    \partial^{\eta\etab}_\gamma\kappa^g_{\text{sig}}(\gamma,\tau)_{s,t} 
\end{pmatrix} \in \mathbb R^4
$$
satisfies the following linear system of hyperbolic PDEs
\begin{align}\label{eqn:system}
\frac{\partial^2}{\partial_s \partial_t}
 \mathbf K_{s,t}^{\eta\etab}(\gamma,\tau)
&= 
\begin{pmatrix}
 \mathbb A_{s,t}(\gamma,\tau)  & 0 & 0 & 0
 \\
 \mathbb A_{s,t}^\eta(\gamma,\tau) & \mathbb A_{s,t}(\gamma,\tau) & 0 & 0\\
 \mathbb A_{s,t}^\etab(\gamma,\tau) & 0 & \mathbb A_{s,t}(\gamma,\tau) & 0 \\
 \mathbb A_{s,t}^{\eta \etab}(\gamma,\tau) & \mathbb A_{s,t}^\etab(\gamma,\tau) & \mathbb A_{s,t}^\eta(\gamma,\tau) & \mathbb A_{s,t}(\gamma,\tau)
\end{pmatrix}
 \mathbf K_{s,t}^{\eta\etab}(\gamma,\tau),
\end{align}
with boundary conditions
\begin{equation*}
    \mathbf K_{0,t}^{\eta\etab}(\gamma,\tau) = \mathbf K_{s,0}^{\eta\etab}(\gamma,\tau) = (1,0,0,0)^\top.
\end{equation*}
\begin{remark}
    In the case where $g=id$ the system of PDEs \eqref{eqn:system} simplifies as
    $\mathbb A^\eta_{s,t}(\gamma,\tau) = \mathbb A_{s,t}(\eta,\tau)$ and $\mathbb A^{\eta\etab}_{s,t} = 0$. For more general static kernels $g$, the derivative $\mathbb A$ might not be available in close form. In that case we can approximate them by finite difference as follows
    \begin{align*}\mathbb A_{s,t}(\gamma,\tau)
    &  = \left\langle\partial_t\gamma^g_s, \partial_t \tau^g_t \right\rangle_{\cH_g} \\
    & \approx \left\langle \frac{1}{\Delta s} (\gamma^g_s - \gamma^g_{s-\Delta s}), \frac{1}{\Delta t} (\tau^g_t - \tau^g_{t- \Delta t})\right\rangle_{\cH_g}\\
    & = \frac{1}{\Delta s\Delta t}  \left(\left\langle \gamma^g_s, \tau^g_t \right\rangle_{\cH_g}  - \left\langle \gamma^g_s, \tau^g_{t-\Delta t} \right\rangle_{\cH_g}  \left\langle \gamma^g_{s-\Delta s}, \tau^g_t \right\rangle_{\cH_g} + \left\langle \gamma^g_{s-\Delta s}, \tau^g_{t-\Delta t} \right\rangle_{\cH_g} \right)\\
    &:= \frac{1}{\Delta s\Delta t}  \Big(\left\langle g(\gamma_s, \cdot), g(\tau_t,\cdot) \right\rangle_{\cH_g}  - \left\langle g(\gamma_s, \cdot), g(\tau_{t-\Delta t}, \cdot) \right\rangle_{\cH_g}  \\
    & \hspace{2cm}-\left\langle g(\gamma_{s-\Delta s}, \cdot), g(\tau_t, \cdot) \right\rangle_{\cH_g} + \left\langle g(\gamma_{s-\Delta s, \cdot)}, g(\tau_{t-\Delta t}, \cdot) \right\rangle_{\cH_g} \Big)\\
    &= \frac{1}{\Delta s\Delta t} \Big( g(\gamma_s, \tau_t) - g(\gamma_s, \tau_{t-\Delta t}) - g(\gamma_{s-\Delta s}, \tau_t) +  g(\gamma_{s-\Delta s}, \tau_{t-\Delta t}) \Big),
    \end{align*}
     where the last equality follows by the reproducing property of the kernel $g$. Similarly for $\mathbb A^\eta$ and $\mathbb A^{\eta,\etab}$
\end{remark}

\subsection{Universal product kernels} 

The solutions to our PPDEs are defined on the product space~$\Omega \subset \bT\times \bR^d \times \BV_0$, see~\eqref{eq:def_Omega}. Therefore, we will need to consider product kernels indexed on such product spaces. The following result provides a simple way of constructing kernels that are universal on a product space as products of kernel on the individual factor spaces.

\begin{lemma}\cite[Lemma 5.2]{blanchard2011generalizing}\label{lemma:prod_kernel}
    Let $\kappa_1,\kappa_2$ be two cc-universal kernels indexed on $\Omega_1, \Omega_2$ respectively. Then the product kernel defined for any $\omega_1,\omega_1' \in \Omega_1$ and $\omega_2,\omega_2' \in \Omega_2$ as
    $$\kappa((\omega_1,\omega_2),(\omega_1',\omega_2')) := \kappa_1(\omega_1,\omega_1')\kappa_2(\omega_2,\omega_2')$$
    is cc-universal on $\Omega_1 \times \Omega_2$.
\end{lemma}
Since $\bT \times \mathbb R^d \subset \mathbb R^{d+1} $, we consider product kernels~$\widetilde \kappa^{k,g}$ defined for any $(s,x,\gamma), (t,y,\tau) \in \bT\times \bR^d \times \BV_0$ as
\begin{equation*}
     \widetilde\kappa^{k,g}((s,x,\gamma), (t,y,\tau)) := k((s,x),(t,y))\kappa^g_{\text{sig}}(\gamma,\tau),
\end{equation*}
where $k : \mathbb R^{d+1} \times \mathbb R^{d+1} \to \mathbb R$ is some cc-universal kernel on~$\mathbb R^{d + 1}$ such as the RBF kernel in~\eqref{eqn:rbf_kernel}, and~$\kappa^g_{\text{sig}}$ is a~$g$-lifted signature kernel (Definition \ref{def:kappag}).

However, signature kernels are only universal when restricted to classes of paths where the signature is an injective map. One such space is~$\BV_{0}^0 \subset \BV_0$, the subset of continuous bounded variation paths that are time-augmented and started at the origin $0$, endowed with the $1$-variation topology.  Denote by~$\mathcal{H}_{\kappa^g_{\text{sig}}}$  the RKHS associated to the signature kernel $\kappa^g_{\text{sig}}$. The next result states that, when restricted to a compact subset $B \subset \BV_{0}^0$, elements of $\mathcal{H}_\kappa|_B$ are dense in $\cC(B)$ with the topology of uniform convergence. In other words, $\kappa^g_{\text{sig}}$ is cc-universal on $\BV_{0}^0$.

\begin{lemma}\label{lemma:univ}
    Let $B \subset \BV_{0}^0$ be compact in $1$-variation. Then, $\mathcal{H}_\kappa|_{B}$ is a dense subset of $\cC(B)$ with the topology of uniform convergence. 
\end{lemma}

The proof of this statement is analogous to \cite[Proposition 3.3]{cass2024weighted} with the difference that the point-separation assumption in the Stone-Weirestrass theorem is justified by noting that $S(\gamma) = S(\tau) \iff \gamma = \tau$ when $\gamma, \tau \in \BV_0^0$ .

It is important to note that the above result might no longer hold if one considers compact subsets of~$\BV_0$ that are not time-augmented and that do not share a common origin. This is due the fact that the signature map is no longer injective on~$\BV_0$, thus the point-separation assumption used in the proof of \Cref{lemma:univ} invoking the Stone-Weirestrass Theorem no longer holds. However, in this paper we are interested in paths that might not have a common starting point. To remedy this issue, we modify the signature kernel $\kappa^g_{\text{sig}}$ and, for any~$s,t\in\bT$, define the kernel $\kappa^{g,\ell}_{\text{sig}} : \BV_0 \times \BV_0 \to \mathbb R$ as
\begin{equation}\label{eq:kappagell}
    \kappa^{g,\ell}_{\text{sig}}(\gamma,\tau) = \ell(\gamma_0, \tau_0)\kappa^g_{\text{sig}}(\gamma-\gamma_0, \tau - \tau_0)
\end{equation}
where $\ell : \mathbb R^e \times \mathbb R^e \to \mathbb R$ is yet another kernel that keeps track of the starting points of the two input paths. 
The signature kernel~$\kappa^g_{\text{sig}}$ is cc-universal on~$\BV_0$ by~\Cref{lemma:univ} hence, if~$\ell$ is cc-universal, \Cref{lemma:prod_kernel} yields that~$\kappa^{g,\ell}_{\text{sig}}$ is cc-universal on $\BV_0$. 
We thus define the product kernel~$\kappa$  for any $(s,x,\gamma), (t,y,\tau) \in \bT\times \bR^d \times \BV_0$ as
\begin{align}\label{eqn:product_kernel}
     \kappa((s,x,\gamma), (t,y,\tau)) 
     = k((s,x),(t,y))\kappa^{g,\ell}_{\text{sig}}(\gamma,\tau),
\end{align}
where $k : \mathbb R^{d+1} \times \mathbb R^{d+1} \to \mathbb R$ is a cc-universal kernel on $\mathbb R^{d + 1}$ and~$\kappa^{g,\ell}_{\text{sig}}$ is from~\eqref{eq:kappagell}. 
Note that for a path~$\check\gamma$ which is constant on~$[0,s]$, as we consider in~$\Omega$ \eqref{eq:def_Omega}, we have~$\check\gamma-\check\gamma_0=\check\gamma-\check\gamma_s$ such that~$(\check\gamma-\gamma_0)\lvert_{[s,T]}$ is a path starting at zero. 
Moreover, because the signature is invariant to translation, 
there exists a measurable map~$\tilde{\kappa}^g_{s}:\BV_s\to\bR$ such that for any $\gamma,\tau\in\BV_0^0 $ with $\check\gamma$ constant on $[0,s]$ one has
$$
\kappa^g_{\text{sig}}(\gamma, \tau )
=
\langle S(\gamma^g)_{0,T}, S(\tau^g)_{0,T} \rangle 
= \langle S(\gamma^g)_{0,s}\cdot S(\gamma^g)_{s,T}, S(\tau^g)_{0,T} \rangle 
=
\tilde{\kappa}^g_{s}(\gamma\lvert_{[s,T]}, \tau).
$$
This entails~$(s,x,\gamma)\mapsto\kappa((s,x,\gamma),(t,y,\tau))\in \cC(\Omega)$, and it is clear what the Gateaux derivative~\eqref{eq:Pathwise_restricted_def} means in this context.
For future reference and technical reasons, we make the following assumptions on these kernels, satisfied for instance by RBF kernels. Denote by $\mathcal{H}_\ell,\mathcal{H}_k$ the RKHS associated to the kernels~$\ell,k$ respectively, and denote by $\varphi:\bR^e \to \mathcal{H}_\ell$ and~$\chi:\bR^{d+1}\to \mathcal{H}_k$ their feature maps, defined as $\varphi(x) := \ell(x, \cdot)$ and $\chi(x) := k(x, \cdot)$.
\begin{assumption}\label{assu:featuremaps}
The feature maps~$\varphi$ and~$\chi$ are bounded, Lipschitz continuous, twice differentiable and with bounded, Lipschitz continuous derivatives, with bounds and Lipschitz constants~$C_\varphi$ and~$C_\chi$ respectively.
\end{assumption}
For the convergence of the numerical scheme to take place, we have to work on a compact subspace of~$(\BV_0,\norm{\cdot}_{p-var})$, for~$p>1$, on which the~$1$-variation norm is bounded. This results in all the estimates of Lemmas~\ref{lemma:bound_sig_appendix} and~\ref{lemma:bound_kappag_appendix} to be bounded and enables the use of Arzel\`a-Ascoli theorem. 
As we recalled in the introduction, for all~$R>0$ and~$\alpha>0$, the set
\begin{equation}
    B(R):=\{ \gamma\in\BV_0: \norm{\gamma}_{1-var,\alpha}\le R\}
\end{equation}
is a compact subset of~$(\BV_0,\norm{\cdot}_{p-var})$ for all~$p>1$ and hence for any~$t\in\bT$ the sets
\begin{equation}\label{eq:BtRcompact}
    B_t(R):=\{ \gamma\in\BV_0: \norm{\gamma}_{1-var,\alpha}\le R,\, \gamma_s=\gamma_0\text{   for all  } s\le t\}
\end{equation}
are compact subsets of~$\{ \gamma\in\BV_0: \gamma_s=\gamma_0\text{   for all  } s\le t\}$ with respect to~$\norm{\cdot}_{p-var}$. Compact sets of~$\Omega$ with respect to~$\bm{d}$ are thus typically of the form
\begin{equation}
\cX =\Big\{(t,x,\gamma)\in \bT\times  \bigcap_{i=1}^d [a_i,b_i] \times\BV_0: \gamma\in B_t(R) \Big\}
\end{equation}
for~$a_i,b_i\in\bR,\,i=1,\cdots,d$ and~$R>0$.

We denote by~$\cH$ the RKHS of the kernel $\kappa:\cX\times\cX\to\bR$ restricted to act on~$\cX$. We define the feature map $\Phi:\cX\to \cH$ as~$
\Phi(\omega)=\kappa(\omega,\cdot)$. 
Since the two kernels $h$ and $\kappa^{g,\ell}_{\text{sig}}$ are continuously twice differentiable, as we have seen in the previous section, so is their product. This implies that for each $\omega\in\cX$, the map $\kappa(\omega,\cdot)$ belongs to $\cC^{\infty,\infty,2}(\cX)$.

Moreover, as a consequence of Lemmas~\ref{lemma:prod_kernel} and~\ref{lemma:univ}, the product kernel~$\kappa$ is cc-univeral on~$\cX\times\cX$. In other words, $\cH$ is dense in~$\cC(\cX)$ with respect to the topology of uniform convergence.



\section{Main theoretical results}\label{sec:theory}
This section aims at answering the questions posed in~\Cref{enum:questions} and restricts our domain of study to pathwise derivatives of order two or lower. This framework is relevant for the backward Kolmogorov equations we presented in Section~\ref{subsec:examples} and can rely on the equations derived in~\Cref{prop:der_kappag}. More precisely we fix~$\bar{N}\in\bN$ and consider the operator~$\cL_t$ in~\eqref{eq:def_cL_BV} with
\begin{equation*}
    \cN:=\{n=(n_1,n_2,n_3)\in\bN_0^3: n_1\le \bar{N},n_2\le \bar{N},n_3\le 2\}.
\end{equation*}


\subsection{Optimal recovery and well-posedness}\label{sec:or-wp}

This section aims at answering the questions (1) to (3) of Section~\ref{subsec:method} regarding the well-posedness of the optimisation problem~\eqref{eq:OptiRecovery}. Our first result verifies that elements of the RKHS $\cH$ are indeed twice differentiable, and therefore belong to the domain of the PPDE operator~$\cL_t$. This answers Question (1) in~\ref{enum:questions}. The proof builds on the estimates derived in Lemma~\ref{lemma:bound_kappag_appendix}. The technical hurdle consists in interchanging the infinite series and the derivative operator. As mentioned earlier, this operation requires the restriction of~$\cH$ to a compact subset of~$\Omega$ called~$\cX$. 
Even though~$\cX$ is closed, the time derivative of~$h\in\cH$ is well-defined since it only perturbs to the right and we set $\partial_t u(T,x,\gamma)=0$. Moreover the spatial and pathwise derivatives are also well-defined over~$\bT\times\bigcap_{i=1}^d[a_i,b_i]\times B_t(R)$, where $B_t(R)$ is defined in~\eqref{eq:BtRcompact}, since the latter is included in an open set of~$\Omega$ and the product kernel~$\kappa$~\eqref{eqn:product_kernel} is defined over~$\bT\times \bR^d \times \BV_0$. We also recall that the pathwise derivatives $\wpar_\gamma^{\bm{\eta}}$ are defined in~\eqref{eq:Pathwise_restricted_def}.
\begin{proposition}\label{proposition:C2}
Let Assumptions~\ref{assu:statickernel} and~\ref{assu:featuremaps} hold for~$g$ and~$k,\ell$ respectively.
For all $h\in\cH$ such that $h=\sum_{i=1}^\infty \alpha_i \kappa(\omega^i,\cdot)$, with $(\alpha_i)_{i\in\bN}$ in~$\bR$ and~$(\omega^i)_{i\in\bN}$ in $\cX$, for all~$\partial\in\{\partial_t^{n_1} \partial_x^{n_2} \wpar_\gamma^{\bm{\eta}}:(n_1,n_2,n_3)\in\cN,\,\bm{\eta}\in\BV_0^{\otimes n_3}\}$ and $\omega=(t,x,\gamma)\in\cX$, we have
\begin{align*}
\partial h(\omega) = \sum_{i=1}^\infty \alpha_i \partial \kappa(\omega^i,\omega).
\end{align*}
In particular the domain of~$\cL_t$ is included in~$\cH$ for all~$t\in[0,T)$.
\end{proposition}

\begin{proof}
Let $h\in\cH$ such that $h=\sum_{i=1}^\infty \alpha_i \kappa(\omega^i,\cdot)$, with $(\alpha_i)_{i\in\bN}$ in~$\bR$ and~$(\omega^i)_{i\in\bN}$ in $\cX$. We prove the claim only for the pathwise derivative~$\partial\in\{\wpar^\eta_\gamma,\wpar^{\eta\etab}_\gamma\}$, with~$\eta,\etab \in \BV_0$, and note that the other cases can be proved in an identical fashion, as the RBF kernel $k$ is infinitely many times differentiable and its derivatives are uniformly bounded. An important observation due to \eqref{eq:Pathwise_restricted_def} is that for all~$(s,x,\gamma), (t,y,\tau) \in \cX$ we have
$$
\wpar^\eta_\gamma \kappa((s,x,\gamma), (t,y,\tau)) = k((s,x),(t,y))\partial^{\eta\lvert_{[t,T]}}_\gamma \kappa_{\text{sig}}^{g,\ell}(\gamma\lvert_{[t,T]},\tau).
$$
This highlights the link between the two types of pathwise derivatives~$\wpar^\eta_\gamma$ and~$\partial^\eta_\gamma$.

$(\bm{\partial^{\eta}_\gamma})$ For all~$n\in\bN$, let~$h_n:=\sum_{i=1}^n \alpha_i \kappa(\omega^i,\cdot)$ hence~$h_n$ tends to~$f$ pointwise and in RKHS norm as~$n\to\infty$. Recall the $\Phi$ was defined as the feature map associated to~$\kappa$, hence~$\Phi(\omega)=\kappa(\omega,\cdot)$. For all~$\omega\in \cX$ and $\eta\in \BV(\bT,\bR^e)$, we have and we define
$$
\partial^\eta_\gamma h_n(\omega) = \sum_{i=1}^n \alpha_i \partial^\eta_\gamma \kappa(\omega^i,\omega), \qquad h'(\omega):= \sum_{i=1}^\infty \alpha_i \partial^\eta_\gamma \kappa(\omega^i,\omega),
$$
and we will prove that~$ \partial^\eta_\gamma h_n$ converges to $h'$ uniformly on $\cX$. Let~$n\in\bN$, the Cauchy-Schwarz inequality allows us to disentangle the two limits:
\begin{align}\label{eq:Cvg_Df}
\abs{\partial^\eta_\gamma h_n(\omega)-h'(\omega) }
= \abs{\sum_{i=n+1}^\infty \alpha_i \partial_\gamma^\eta \kappa(\omega^i,\omega)}
&=\abs{ \sum_{i=n+1}^\infty \alpha_i \lim_{\ep\to0} \frac{\kappa(\omega^i,(t,x,\gamma+\ep\eta)) -\kappa(\omega^i,(t,x,\gamma))}{\ep}} \\
&=\lim_{N\to\infty}\lim_{\ep\to0}\frac{1}{\ep} \abs{\sum_{i=n+1}^N \alpha_i  \Big\langle \Phi(\omega^i), \Phi(t,x,\gamma+\ep\eta)-\Phi(t,x,\gamma) \Big\rangle} \nonumber\\
&=\lim_{N\to\infty}\lim_{\ep\to0}\frac{1}{\ep}\abs{\left\langle \sum_{i=n+1}^N \alpha_i\Phi(\omega^i), \Phi(t,x,\gamma+\ep\eta)-\Phi(t,x,\gamma) \right\rangle} \nonumber\\
&\le \lim_{N\to\infty} \norm{ \sum_{i=n+1}^N \alpha_i \Phi(\omega^i)}_{\cH} \lim_{\ep\to0}\frac{1}{\ep}\norm{\Phi(t,x,\gamma+\ep\eta)-\Phi(t,x,\gamma)}_{\cH}. \nonumber
\end{align}
We observe that
\begin{align*}
    \norm{ \sum_{i=n+1}^\infty \alpha_i \Phi(\omega^i)}_{\cH}^2
    = \left\langle \sum_{i=n+1}^\infty \alpha_i \Phi(\omega^i), \sum_{j=n+1}^\infty \alpha_j \Phi(\omega^j)\right\rangle_{\cH}
    &=\sum_{i=n+1}^\infty \alpha_i  \sum_{j=n+1}^\infty \alpha_j \left\langle \Phi(\omega^i),\Phi(\omega^j)\right\rangle_{\cH}\\
    &= \sum_{i,j=n+1}^\infty \alpha_i\alpha_j \kappa(\omega^i,\omega^j) \\
    &= \norm{h-h_n}_{\cH}^2,
\end{align*}
which tends to zero.
Furthermore,
\begin{align*}
    \lim_{\ep\to0}&\frac{1}{\ep^2}\norm{\Phi(t,x,\gamma+\ep\eta)-\Phi(t,x,\gamma)}_{\cH}^2\\
    &= k((t,x),(t,x)) \lim_{\ep\to0}\frac{1}{\ep^2} \Big( \kappa^{g,\ell}_{\text{sig}}(\gamma+\ep\eta,\gamma+\ep\eta)-2\kappa^{g,\ell}_{\text{sig}}(\gamma+\ep\eta,\gamma) + \kappa^{g,\ell}_{\text{sig}}(\gamma,\gamma) \Big)\\
    &= k((t,x),(t,x)) \bigg(\ell(\gamma_t,\gamma_t)\lim_{\ep\to0}\frac{1}{\ep^2}  \norm{ S((\gamma+\ep\eta)^g) - S(\gamma^g)}_{\overline{T(\cH_g)}}^2\\
    &\qquad+ \kappa_{\text{sig}}^g(\gamma-\gamma_t,\gamma-\gamma_t)\lim_{\ep\to0}\frac{1}{\ep^2}\norm{\varphi(\gamma_t+\ep\eta_t)-\varphi(\gamma_t)}_{\cH_\ell}^2 \bigg)\\
    &= k((t,x),(t,x)) \left(\ell(\gamma_t,\gamma_t)\norm{\partial^\eta_\gamma S(\gamma^g)}_{\overline{T(\cH_g)}}^2 +\kappa_{\text{sig}}^g(\gamma-\gamma_t,\gamma-\gamma_t)\norm{\partial^{\eta_t}_{\gamma_t}\varphi(\gamma_t)}_{\cH_\ell}^2 \right)
\end{align*}
which is uniformly bounded over all $\omega\in\ \cX$ by Lemma \ref{lemma:bound_sig_appendix} and Assumption~\ref{assu:featuremaps} for~$\ell$. We used dominated convergence to pass to the last line which follows from~\eqref{eq:bound_sig_DS0}.
This entails the convergence of $\partial^\eta_\gamma h_n$ towards $h'$ is uniform in $\cX$;
therefore~$h$ is differentiable and $\partial^\eta_\gamma h=h'=\sum_{i=1}^\infty \alpha_i \partial^\eta_\gamma\kappa(\omega^i,\cdot)$. 

$(\bm{\partial^{\eta\etab}_\gamma})$ For all~$\omega\in \cX$ and $\eta,\etab\in \BV(\bT,\bR^e)$, we have and we set
$$
\partial^{\eta\etab}_\gamma h_n(\omega) = \sum_{i=1}^n \alpha_i \partial^{\eta\etab}_\gamma \kappa(\omega^i,\omega), \qquad h''(\omega):= \sum_{i=1}^\infty \alpha_i \partial^{\eta\etab}_\gamma \kappa(\omega^i,\omega),
$$
and we will prove that~$\lim_{n\to\infty} \partial^{\eta\etab}_\gamma h_n=h''$. First note that
\begin{align*}
\partial^{\eta\etab}_\gamma \kappa(\omega^i,\omega)
&= \lim_{\ep\to0} \frac{1}{\ep} \Big( \partial^\etab_\gamma \kappa(\omega^i,(t,x,\gamma+\ep\eta)) - \partial^\etab_\gamma \kappa(\omega^i,(t,x,\gamma))\Big) \\
&= \lim_{\ep\to0}\lim_{\delta\to0} \frac{1}{\ep\delta} \Big( \kappa(\omega^i,(t,x,\gamma+\ep\eta+\delta \etab)) - \kappa(\omega^i,(t,x,\gamma+\ep\eta)) - \kappa(\omega^i,(t,x,\gamma+\delta\etab)) + \kappa(\omega^i,(t,x,\gamma)) \Big)\\
&= \lim_{\ep\to0}\lim_{\delta\to0} \frac{1}{\ep\delta} \Big\langle \Phi(\omega^i), \, \Phi(t,x,\gamma+\ep\eta+\delta \etab)) - \Phi(t,x,\gamma+\ep\eta)) - \Phi(t,x,\gamma+\delta\etab)) + \Phi(t,x,\gamma)\Big\rangle.
\end{align*}
Applying the same approach as for the first derivative, we obtain by Cauchy-Schwarz inequality
\begin{align*}
&\abs{\partial^{\eta\etab}_\gamma h_n(\omega)-h''(\omega) }
=\abs{ \sum_{i=n+1}^\infty \alpha_i \partial^{\eta\etab}_\gamma \kappa(\omega^i,\omega) }\\
&= \lim_{N\to\infty} \lim_{\ep\to0}\lim_{\delta\to0} \frac{1}{\ep\delta} \abs{ \left\langle \sum_{i=n+1}^N \alpha_i \Phi(\omega^i), \, \Phi(t,x,\gamma+\ep\eta+\delta \etab) - \Phi(t,x,\gamma+\ep\eta) - \Phi(t,x,\gamma+\delta\etab) + \Phi(t,x,\gamma)\right\rangle
} \\
& \le \norm{h-h_n}_{\cH} \lim_{\ep\to0}\lim_{\delta\to0} \frac{1}{\ep\delta} \norm{\Phi(t,x,\gamma+\ep\eta+\delta \etab) - \Phi(t,x,\gamma+\ep\eta) - \Phi(t,x,\gamma+\delta\etab) + \Phi(t,x,\gamma)}_{\cH} \\
&=\norm{h-h_n}_{\cH} k((t,x),(t,x))
\bigg(\norm{\varphi(\gamma_t)}_{\cH_\ell} \norm{\partial^{\eta\etab}_\gamma S(\gamma^g)}_{\overline{T(\cH_g)}}
+ \norm{\partial^{\eta_t}_{\gamma_t}\varphi(\gamma_t)}_{\cH_\ell}\norm{\partial^{\etab}_\gamma S(\gamma^g)}_{\overline{T(\cH_g)}} \\
&\qquad \qquad + \norm{\partial^{\etab_t}_{\gamma_t}\varphi(\gamma_t)}_{\cH_\ell}\norm{\partial^{\eta}_\gamma S(\gamma^g)}_{\overline{T(\cH_g)}}
+ \norm{\partial^{\eta_t\etab_t}_{\gamma_t}\varphi(\gamma_t)}_{\cH_\ell}\norm{S(\gamma^g)}_{\overline{T(\cH_g)}}
\bigg),
\end{align*}
where we used dominated convergence as for the first derivative.
Lemma \ref{lemma:bound_sig_appendix} and Assumption~\ref{assu:featuremaps} ensure 
that~$\partial^{\eta\etab}_\gamma h_n$ converges towards $h''$ uniformly in $\cX$ and thus $h$ is twice differentiable with $\partial^{\eta\etab}_\gamma h=h'' = \sum_{i=1}^\infty \alpha_i \partial^{\eta\etab}_\gamma\kappa(\omega^i,\cdot)$.
\end{proof}

\begin{remark}\label{rem:compact_consequences}
    Proposition 4.1 proves that a function $h|_{\mathcal{X}}$ of the RKHS restricted to $\mathcal{X}$ is indeed differentiable. This is not guaranteed on the whole space of bounded variation paths~$\Omega$ as the proof exploits compactness to obtain uniform bounds. This assumption could however be lifted by exploiting weighted spaces as in~\cite{cuchiero2023global} or robust signatures as in~\cite{chevyrev2018signature}.
\end{remark}


Now that we have checked that the constraints of the optimal recovery problem~\eqref{eq:OptiRecovery} make sense, we can solve this optimisation problem explicitly due to a variant of the representer theorem involving higher order derivatives \cite{owhadi2019operator} materialising in \textbf{PDE scheme} just before Remark \ref{remark:2}, which gives an answer to Questions (2) and (3) of our list~\ref{enum:questions}.

\begin{remark}
    The matrix $\widetilde \cK$ is well-defined by the differentiability of the product kernel deriving from assumptions~\ref{assu:statickernel} and~\ref{assu:featuremaps} on ~$g$ and~$k,\ell$
\end{remark}

\begin{remark}
    Note that~$(\alpha_i)_{i=1}^{m+n}$ crucially depend on the collocation points and on~$m$ and~$n$; in particular they have no reason to remain stable for different values of~$m,n$. This makes comparison among the terms of the sequence~$(u_{m,n})_{m,n}$ much trickier. As another avenue for future research, we note that orthogonalising the collocation points in a way that the~$(\alpha_i)_i$ remain constant for all~$m,n$ could lead to better control of the norm.
\end{remark}

\subsection{Consistency}

For any subset~$\cN\in\bN^3_0$, let us define the $\cC^{\cN}(\cX)$ norm as 
$$
\norm{h}_{\cN}:= 
\sum_{(n_1,n_2,n_3)\in\cN} \sup \Big\{\abs{\partial_t^{n_1} \partial_x^{n_2} \partial_\gamma^{\bm{\eta}} h(t,x,\gamma)}: (t,x,\gamma)\in\cX,\bm{\eta}\in \BV^{\otimes n_3}_0,\, \norm{\bm{\eta}}_{1-var;\bT} \le 1 \Big\} .
$$
Due to the fact that our estimates in~\Cref{lemma:bound_sig_appendix} do not hold for all~$n\in\bN$, we restrict to the specification needed in our examples
\begin{equation}\label{eq:cN_spec}
\cN= \{(n_1,n_2,n_3)\in \bN^3_0: n_1\le 1, n_2=n_3=0, \text{ or  } n_1=0, n_2+n_3\le 2\}.    
\end{equation}
This abstract definition is simply a way of restricting the range of derivatives to~$\mathscr{D}:=\{{\rm id}, \partial_t,\partial_x,\partial^2_x,\partial^\eta_\gamma\partial_x, \partial_\gamma^\eta, \partial^{\eta\etab}_{\gamma}\}$. These are precisely the partial derivatives appearing in the parabolic PPDEs of interest (see Section~\ref{subsec:examples}). In particular it excludes derivatives of order higher than three and crossed derivatives such as~$\partial_{tx}$. 
Therefore the norm reads
\begin{equation}\label{eq:def_norm_CN}
\norm{h}_{\cN}:= 
\sum_{\partial\in\mathscr{D}} \sup \Big\{\abs{\partial h(\omega)}: \omega\in\cX,\eta,\etab\in \BV_0,\, \norm{\eta}_{1}\vee\norm{\etab}_{1-var} \le 1 \Big\} .
\end{equation} 
Our main result, inspired by~\cite[Theorem 1.2]{chen2021solving}, follows. Its proof is given at the end of the section. 
\begin{theorem}\label{th:MainCvg}
Let Assumptions~\ref{assu:statickernel} and~\ref{assu:featuremaps} hold for~$g$ and~$k,\ell$ respectively. 
Let~$\cN$ be as in~\eqref{eq:cN_spec} with~$\eta,\etab\in\BV_0$ and assume either one of the following conditions holds: 
\begin{enumerate}[(i)]
    \item The family~$(u_{m,n})_{m,n\in\bN}$ is relatively compact in~$\cC^{\cN}(\cX)$ and there exists a unique solution $u^\star$ to the PPDE~\eqref{eq:mainPDE} in~$\cC^{\cN}(\cX)$; 
    \item There exists a unique solution $u^\star$ to the PPDE~\eqref{eq:mainPDE} in~$\cH$.
\end{enumerate} 
If, moreover, as $m$ and $n$ tend to infinity, the collocation points~$(\omega^i)_{i=1}^{m}$ and~$(\omega^i)_{i=m+1}^{m+n}$ form a dense family of $\cX^\circ$ and~$\partial\cX$ respectively, then $(u_{m,n})_{m,n\in\bN}$ converges towards~$u^\star$ in the $\cC^{\cN}$-topology as~$m,n$ tend to~$+\infty$.
\end{theorem}
The assumptions of the theorem all deserve separate discussion, studied in reverse order.

\textbf{On the collocation points.} 
It is natural to ask the collocation points to be dense in~$\cX$, however this leaves a lot of flexibility for choosing them in an optimal way, especially in the space of paths. This issue is directly related to the rate of convergence of the method which lies beyond the scope of this paper. In practice though, one is only interested in covering the support of the evaluation points. During the course of our experiments, we found that sampling the collocation points randomly (as Brownian motion trajectories) led to similar accuracy as sampling them from the same distribution as the evaluation points. 

\textbf{On the well-posedness of~\eqref{eq:mainPDE}}
 Existence and uniqueness of~\eqref{eq:mainPDE} is discussed in several examples of interest in Section~\ref{subsec:examples}, and in the more general Volterra case in~\cite{bonesini2023rough}.
Uniqueness allows us to show that all the converging subsequences of~$(u_{m,n})_{m,n}$ have the same limit, thus implying the sequence does converge. 
Unfortunately, known uniqueness results hold on a subset of~$\cC(\tom)$ whereas we require uniqueness in (a subset of)~$\cC(\cX)$ where~$\cX$ is strictly included in~$\tom$. Notice that the solution to the PPDE defined over the space $\tom$ is also solution over the subspace $\mathcal{X}$, however it may not be unique over the subspace anymore.

We mentioned in Remark \ref{rem:compact_consequences} that the assumption that $\mathcal{X}$ is a compact subspace of $\Omega$ could be lifted. At the cost of increased technicality one can also replace the space of bounded variation paths with a space of arbitrary $p$-variation paths. Uniqueness of the solution to the PPDE should then be attainable by classical probabilistic techniques (i.e. Feynman--Kac theorem).

\textbf{On the assumption~$u^\star \in \cH$.}
\vspace{-.1cm}
\begin{itemize}
    \item Proving convergence of~$(u_{m,n})_{m,n}$ at the very least requires to find a norm in which this family is uniformly bounded.
    The reason for assuming~$u^\star \in \cH$ is that it implies both the RKHS and~$\cC^{\cN}$ norms of~$(u_{m,n})_{m,n}$ are uniformly bounded, which in turn yields the relative compactness of this family in~$\cC^{\cN}$, by Lemma~\ref{lemma:compact_H}. It is still slightly different from condition (i) because uniqueness holds in a different space. 
    \item If $u^\star\notin\cH$ then the RKHS norm cannot be expected to be bounded and we lose our most promising tool to derive estimates for other norms. Note that we do not need equicontinuity if we look for compactness in~$\bR$, but we still require~$u_{m,n}$ and their derivatives to be bounded uniformly in~$\bN^2$.
    \item For $u^\star$ to be in~$\cH$ would require at minima to be in~$\cC^\infty$. In this direction, Theorem 3.10 in~\cite{dupire2022functional} shows how to expand a path (semimartingale) functional with respect to terms of the signature. It is conceivable that such a representation also holds in the fractional setting for sufficiently smooth functionals, which could then be identified to an element of the signature RKHS. When the underlying process is a semimartingale, smoothness of the payoff function (and the coefficients) essentially ensures smoothness of the conditional expectation (i.e. the solution to the Kolmogorov equation). However when the direction of the pathwise derivative is singular (only square integrable) one can only prove twice differentiabiliy of the solution (see \cite[Proposition 2.23]{bonesini2023rough}). This remains a glass ceiling until one unveils how to exploit the regularisation properties of the expectation.
\end{itemize}


\begin{remark}
    Several other works in the literature assume that $u^\star\in\cH$. The authors of~\cite{chen2021solving} make in addition the classical assumption that $\cH$ is embedded in a Sobolev space, a condition we drop because we are able to prove sufficient regularity only with the help of the bounded RKHS norm. In finite dimensions, this condition allows to prove convergence rates, see \cite[Proposition 11.30]{wendland2004scattered}. See also \cite{batlle2023error} who provide error bounds under additional assumptions. 
    
\end{remark}

The proof strategy for~\Cref{th:MainCvg} using condition (ii) consists in showing that~$(u_{m,n})$ is relatively compact and then extracting converging subsequences. Hence before presenting it we need to identify compact subsets of $\cH$ with respect to the $\cC^{\cN}$-topology.
The RKHS norm is linked to the regularity of the function, therefore a family of~$\cH$ bounded in RKHS norm is equicontinuous. If this family is also bounded in~$\cC^{\cN}$ norm we can conclude by Arzelà-Ascoli theorem that it is relatively compact.
\begin{lemma}\label{lemma:compact_H}
Let Assumptions~\ref{assu:statickernel} and~\ref{assu:featuremaps} hold for~$g$ and~$k,\ell$ respectively.
Let~$\cY\subset \cH$ be bounded under both the RKHS and $\cC^{\cN}$ norms, then~$\cY$ is relatively compact with respect to the $\cC^{\cN}$ topology.
\end{lemma}
\begin{proof}
We will prove sequential compactness thanks to Arzelà--Ascoli theorem. 
Consider a sequence of functions $(h_n)_{n\in\bN}\subset\cY$, bounded under the $\cC^{\cN}$-topology and the RKHS norm by a constant $C_h$. 
We consider~$\omega=(t,x,\gamma),\omegab=(t,x,\gammab)\in \cX$, and we have by the reproducing property, Cauchy-Schwarz inequality and~\eqref{eq:bound_sig_DS0} that
\begin{align*}
    \abs{h_n(\omega)-h_n(\omegab)} 
    &= \abs{\langle h_n, \Phi(\omega)-\Phi(\omegab) \rangle_\cH} \\
    &\le \norm{h_n}_{\cH} \norm{\Phi(\omega)-\Phi(\omegab)}_\cH\\
    &\le C_h  \sqrt{k((t,x),(t,x)) } \left( \norm{\varphi(\gamma_t)}_{\cH_\ell}\norm{ S(\gamma^g) - S(\gammab^g)}_{\overline{T(\cH_g)}}
    +\norm{\varphi(\gamma_t)-\varphi(\gammab_t)}_{\cH_\ell}\norm{S(\gammab^g)}_{\overline{T(\cH_g)}}\right)\\
    &\le C_h C_\varphi \sqrt{k_\infty} \norm{\gamma-\gammab}_{0;[t,T]} \left(\overline{DS_0}(\abs{\gamma}_{1-var;[t,T]},\abs{\gammab}_{1-var;[t,T]})+\overline{S_0}(\abs{\gammab}_{1-var;[t,T]})\right)
\end{align*}
where~$k_\infty:=\sup_{(t,x,\gamma)\in\cX} k((t,x),(t,x))<+\infty$ and~$C_\varphi<\infty$ by Assumption~\ref{assu:featuremaps}. 
Since $\overline{S_0}$ and~$\overline{DS_0}$ are continuous and $\abs{\gamma}_{1-var;[t,T]},\abs{\gammab}_{1-var;[t,T]}$ are bounded by $R$, there exists a constant $C>0$, independent of $\omega$, such that $\abs{h_n(\omega)-h_n(\omegab)}\le C \norm{\gamma-\gammab}_{0;[t,T]}\le C \norm{\gamma-\gammab}_{p-var;[t,T]}$. Therefore, $\{h_n\}_{n\in\bN}$ is equicontinuous on the compact $\cX$ as it is equipped with the $p$-variation norm.

Following the approach of \eqref{eq:Cvg_Df} and leveraging on~\eqref{eq:bound_sig_DS1} and~\eqref{eq:bound_sig_DS2}, similar bounds hold for $h_n$'s derivatives with $\eta,\etab\in \BV(\bT,V)$. Similarly to the computations in the proof of Proposition~\ref{proposition:C2}, Assumption~\ref{assu:featuremaps} and Lemma~\ref{lemma:bound_sig_appendix} yield
\begin{align*}
    &\abs{\partial^\eta_\gamma h_n(\omega)-\partial^\eta_\gamma h_n(\omegab)} \\
&\le C_h  k_\infty \bigg(\norm{\varphi(\gamma_t)}_{\cH_\ell} \norm{\partial^{\eta}_\gamma S(\gamma^g)-\partial^{\eta}_\gamma S(\gammab^g)}_{\overline{T(\cH_g)}}
+ \norm{\partial^{\eta_t}_{\gamma_t}\varphi(\gamma_t)}_{\cH_\ell}\norm{ S(\gamma^g)- S(\gammab^g)}_{\overline{T(\cH_g)}} \\
&\qquad \qquad + \norm{\varphi(\gamma_t)-\varphi(\gammab_t)}_{\cH_\ell}\norm{\partial^{\eta}_\gamma S(\gamma^g)}_{\overline{T(\cH_g)}}
+ \norm{\partial^{\eta_t}_{\gamma_t}\varphi(\gamma_t)-\partial^{\eta_t}_{\gamma_t}\varphi(\gammab_t)}_{\cH_\ell} \norm{S(\gamma^g)}_{\overline{T(\cH_g)}}
\bigg)\\
&\le C_h C_\varphi k_\infty \norm{\gamma-\gammab}_{0;[t,T]} \Big(\overline{DS_1}(\abs{\gamma}_{1},\abs{\gammab}_{1},\norm{\eta}_{1})+\overline{DS_0}(\abs{\gamma}_{1},\abs{\gammab}_{1})+\overline{S_1}(\abs{\gamma}_{1},\norm{\eta}_{1})+\overline{S_0}(\abs{\gamma}_{1}),
\end{align*}
where we abbreviated the norms to~$\abs{\gamma}_1=\abs{\gamma}_{1-var;[t,T]}$.
Finally, in the same manner one obtains
\begin{align*}
    &\abs{\partial^{\eta\etab}_\gamma h_n(\omega)-\partial^{\eta\etab}_\gamma h_n(\omegab)} \\
    &\le C_h C_\varphi k_\infty \norm{\gamma-\gammab}_{0;[t,T]} \Big(\overline{DS_2}(\abs{\gamma}_{1},\abs{\gammab}_{1},\norm{\eta}_{1},\norm{\etab}_1)
    +\overline{DS_1}(\abs{\gamma}_{1},\abs{\gammab}_{1},\norm{\eta}_{1})\\
    &+\overline{DS_1}(\abs{\gamma}_{1},\abs{\gammab}_{1},\norm{\etab}_{1})
    +\overline{DS_0}(\abs{\gamma}_{1},\abs{\gammab}_{1})
    +\overline{S_2}(\abs{\gamma}_{1},\norm{\eta}_{1},\norm{\etab}_1)
    +\overline{S_2}(\abs{\gammab}_{1},\norm{\eta}_{1},\norm{\etab}_1)\\
    &+\overline{S_1}(\abs{\gamma}_{1},\norm{\eta}_{1})
    +\overline{S_1}(\abs{\gammab}_{1},\norm{\eta}_{1})
    +\overline{S_1}(\abs{\gamma}_{1},\norm{\etab}_{1})
    +\overline{S_1}(\abs{\gammab}_{1},\norm{\etab}_{1})
    +\overline{S_0}(\abs{\gamma}_{1})
    +\overline{S_0}(\abs{\gammab}_{1})\Big).
\end{align*}
By the continuity of $\overline{S_0},\overline{S_1},\overline{S_2},\overline{DS_0},\overline{DS_1}$ and $\overline{DS_2}$, we can also conclude that the families $(\partial^{\eta}_\gamma h_n)_{n\in\bN}$ and $(\partial^{\eta\etab}_\gamma h_n)_{n\in\bN}$ are equicontinuous on the compact~$\cX$ with respect to the $p$-variation norm. The case where~$\ell(\gamma_t,\gammab_t)\neq1$ would yield the same result by exploiting the boundedness and Lipschitz continuity of~$\varphi$ and its derivatives, and the estimates from Lemma~\ref{lemma:bound_sig_appendix}.

We only give details of the equicontinuity with respect to paths since the counterpart on~$\bT\times\bR^d$ can be proved in a similar but easier fashion, using that the feature map associated to~$k$ and its two derivatives are bounded and Lipschitz continuous. Furthermore, equicontinuity of the derivatives with respect to time and space is also straightforward with the same arguments.

Coupled with the uniform~$\cC^{\cN}$ bounds, Arzelà-Ascoli's theorem implies that $(\partial h_n)_{n\in\bN}$ is relatively compact in~$\cC(\cX)$ for all~$\partial\in\mathscr{D}$. In particular, there exists a subsequence~$(n_i)_{i\in\bN}$ such that~$h_{n_i}$ converges as~$i$ goes to~$+\infty$; then $(\partial_t h_{n_i})_{i\in\bN}$ is a subset of a relatively compact space hence it is one itself and we can find a subsubsequence~$n_{i_j}$ such that~$\partial_t h_{n_{i_j}}$ converges. We can go on for each further derivative until we have found a subsequence~$(n_l)$ for which~$\partial h_{n_l}$ converges in~$\cC(\cX)$ converges for all~$\partial\in\mathscr{D}$. This is precisely convergence in~$\cC^{\cN}$ with respect to the norm defined in~\eqref{eq:def_norm_CN}.
\end{proof}
\begin{proof}[Proof of Theorem \ref{th:MainCvg}]
Assume that condition (ii) holds. 
Since~$u^\star$ solves the PPDE~\eqref{eq:mainPDE} with the Fréchet derivative~$\widetilde{\partial}^\eta$ and that the Gateaux derivative~$\widehat{\partial}^\eta$ is a weaker type, $u^\star$ also solves the weaker PPDE~\eqref{eq:BV_PPDE}. In particular, $u^\star$ satisfies the constraints of the optimal recovery problem~\eqref{eqn:exact_sol} at every point in~$\cX$, and since $u_{m,n}$ is the minimiser, we must have~$\norm{u_{m,n}}_{\cH} \le \norm{u^\star}_{\cH}<\infty$. 
Moreover, for all~$\partial\in\mathscr{D}$ and~$\eta,\etab\in\BV_0$ such that~$\max(\norm{\eta}_{1-var;[0,T]},\norm{\etab}_{1-var;[0,T]})\le1$, Cauchy-Schwarz inequality and the same calculations as in the proof of~\Cref{proposition:C2} yield
    \begin{align*}
        \sup_{\omega\in\cX} \abs{\partial u_{m,n}(\omega)}
        = \sup_{\omega\in\cX} \abs{\partial\langle u_{m,n}, \kappa(\omega,\cdot)\rangle_{\cH} }
        \le \sup_{\omega\in\cX} \norm{\partial\kappa(\omega,\cdot)}_{\cH} \norm{u_{m,n}}_{\cH}.
    \end{align*}
    We take evaluation points over a compact and $\partial\kappa(\omega,\cdot)$ is continuous hence $\sup_{\omega\in\cX} \norm{\partial\kappa(\omega,\cdot)}_{\cH}^2 <\infty$. Thus there exists~$C>0$ such that for all such~$\eta,\etab$,
    $$
    \sup_{\omega\in\cX} \abs{\partial u_{m,n}(\omega)} \le C \norm{u^\star}_{\cH}, \quad \text{for all  } \partial\in\mathscr{D}.
   $$
    This proves that $(u_{m,n})_{m,n\in\bN}$ is bounded under both the RKHS and the $\cC^{\cN}$ norms and therefore this family form a relatively compact space by Lemma \ref{lemma:compact_H}. This means that for each sequence~$(u_{m_k,n_k})_{k\in\bN}$ included in~$(u_{m,n})_{m,n\in\bN}$, where~$n_k,m_k$ diverge to infinity as~$k$ goes to infinity,
    there exists a subsequence, also denoted $(u_{m_k,n_k})_{k\in\bN}$ for conciseness, which converges in $\cC^{\cN}$ to a limit $u_\infty$, as $k$ goes to infinity. Since~$(u_{m_k,n_k})_{k\in\bN}$ is bounded under the RKHS norm, we have~$u_\infty\in\cH$.
    
    Under condition (i), the same conclusion holds except that~$u_\infty$ is an element of the closure of~$\cH$ with respect to the~$\cC^{\cN}$ norm. The universality property of the kernel derived in Lemma~\ref{lemma:univ} extends to~$\cC^\cN$ by Proposition~\ref{proposition:C2} and entails that~$\overline{\cH}=\cC^{\cN}(\cX)$.
    Indeed, for all~$h\in\cC^{\cN}(\cX)$ there is a sequence~$(h_k)_{k\in\bN}$ in~$\cH$ that converges to~$h$ in supremum norm. Moreover, for all~$\partial\in\mathscr{D}$ and~$\omega\in\cX$, $\partial h_k(\omega)=\langle h_k, \partial\kappa(\omega,\cdot)\rangle_\cH$ hence~$\partial h_k$ also converges in  supremum norm towards~$\partial h$.
    
    We now need to show that $u_\infty$ solves the PPDE \eqref{eq:mainPDE}. 
Let us define $v := \mathcal{L}_t u_\infty $ and $v_{k} := \mathcal{L}_t u_{m_k,n_k}$ for all $k\in\bN$. For any $\omega=(t,x,\gamma)\in\cX^\circ$, by the triangle inequality and because~$v_{k}(\omega^i)=g(\omega^i)$ for any $1\leq i \leq m_k$, we have
\begin{align}\label{eq:convergence_triangle}
    |v(\omega)-g(\omega)| &\leq \min_{1\leq i \leq m_k} \Big\{\abs{v(\omega)-v(\omega^i)} + \abs{v(\omega^i) - v_{k}(\omega^i)} + \abs{g(\omega^i)-g(\omega)} \Big\}.
\end{align}
Recall that $v$ and $g$ are uniformly continuous over the compact set $\cX$ and, as $k$ goes to infinity, $(\omega^i)_{1\le i \le m_k}$ form a dense family of $\cX^\circ$. Therefore, for all $\ep>0$ there exists $M\in\bN$ such that, if~$m_k\ge M$, we have
$$
\min_{1\leq i \leq m_k} \Big\{\abs{v(\omega)-v(\omega^i)} + \abs{g(\omega^i)-g(\omega)} \Big\} \le \ep.
$$
In addition, $v_{k}$ converges uniformly to $v$ because $u_{m_k,n_k}$ converges to $u_\infty$ in $\cC^{\cN}$. Since~$\ep>0$ was arbitrary, 
Equation \eqref{eq:convergence_triangle} thus entails that $v(\omega)=g(\omega)$.
Following a similar argument it can be shown that $u_\infty(\omega)=\phi(\omega)$ for all $\omega\in\partial\cX$. 

In conclusion, $u_\infty$ is a solution of the PPDE \eqref{eq:mainPDE}. Under condition (i) (respectively (ii)), it belongs to~$\cC^{\cN}$ (respectively $\cH$) and is the unique classical solution in this space, implying that~$u^\star=u_\infty$. Since every convergent subsequence $(u_{m,n})_{m,n\in\bN}$ converges to the same limit $u^\star$, the whole sequence also converges to $u^\star$ in the $\cC^{\cN}$ topology.
\end{proof}

%


\section{Numerical experiments}\label{sec:numerics}

In this section, we present numerical experiments benchmarking our signature kernel PPDE solver formulated as the optimal recovery problem~\eqref{eq:OptiRecovery}, whose exact solution is given by~\eqref{eqn:exact_sol} with coefficients obtained by solving the linear systems~\eqref{eqn:optim_system}--\eqref{eqn:linear_system} against either analytic solutions (when available) or classical Monte Carlo solvers. We consider two examples mentioned in subsection~\ref{subsec:examples}, namely the path-dependent heat equation~\eqref{eqn:PPDEfBM}) and the rough Bergomi PPDE~\eqref{eq:PPDE_rvol}. For both experiments, we consider two types of errors between the predicted prices and true prices, namely the mean squared error (MSE), and the mean absolute error (MAE). The optimal kernel hyperparameters are determined by cross-validation. We will conclude the section with a discussion to compare our kernel approach  with recent neural networks techniques for solving PPDEs. Our code is available at \url{https://github.com/crispitagorico/sigppde}.

\subsection{Fractional Brownian motion}\label{sec:fbm}

In this first example presented in Section~\ref{subsec:example_fbm}, we consider a one dimensional fractional Brownian motion $(\Wh_t=\int_0^t K(t,r)\D W_r)_{t\in\bT}$ with Hurst exponent $H\in(0,\half)$. We recall that~$\Theta^t_s=\bE[\Wh_s|\cF_t]$ for~$s\ge t$. By \cite[Theorem 4.1]{viens2019martingale}, under appropriate regularity assumptions on the functions $\phi,f$, the conditional expectation
$$
\bE\left[\phi(\Wh_T)+\int_t^T f(s,\Wh_s)\ds \lvert \cF_t\right] =: u(t,\Theta^t)
$$
is realised as the solution of the path-dependent heat equation 
$$
\partial_t u(t,\gamma) + \langle \partial^2_\gamma u(t,\gamma),(K^t,K^t)\rangle +f(t,\gamma)=0, \qquad u(T,\gamma) = \phi(\gamma_T).
$$
In our experiments we choose $f\equiv 0$ and consider the following three instances of function $\phi$, for which analytic expressions of the conditional expectation are available:
\begin{enumerate}
    \item $\phi(x)=x, \quad\quad\quad \quad \ \bE[\phi(\Wh_T)|\cF_t] = \Theta_T^t$;
    \item $\phi(x)=\E^{\nu x},\quad\quad \quad \ \  \bE[\phi(\Wh_T)|\cF_t]=\exp\left(\nu\Theta_T^t + \frac{\nu^2 (T-t)^{2H}}{2} \right)$;
    \item $\phi(x) = (x-K)_+, \quad \bE[\phi(\Wh_T)|\cF_t]=\frac{(T-t)^H}{\sqrt{2\pi}} \exp\left(-\frac{(K-\Theta^t_T)^2}{2(T-t)^{2H}}\right) - (K-\Theta^t_T) \vartheta\left(\frac{\Theta^t_T-K}{(T-t)^H}\right)$,
\end{enumerate}
where $\nu, K \in \bR$ and $\vartheta$ is the standard normal cumulative distribution function. 

Because analytic prices are available, we limit ourselves to assess the performance of our kernel algorithm to recover these true prices. Collocation points for the kernel method were chosen uniformly on $[0,1]$ for the time variable and sampled from the process $\Theta^t$ for the path variable. 

\begin{figure}[ht]
  \begin{minipage}[t]{0.33\linewidth}
    \centering
    \includegraphics[scale=0.33]{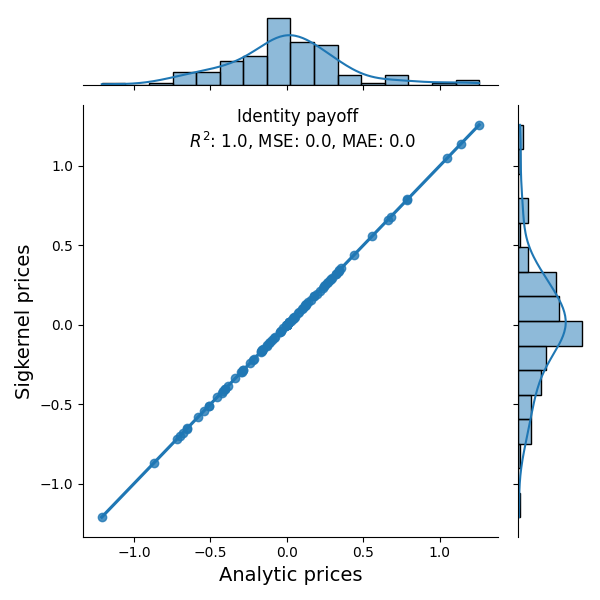}
  \end{minipage}%
  \begin{minipage}[t]{0.33\linewidth}
    \centering
    \includegraphics[scale=0.33]{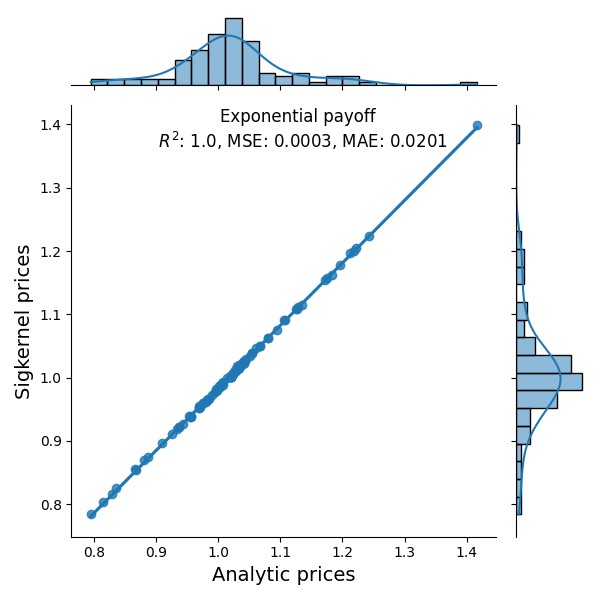}
  \end{minipage}
  \begin{minipage}[t]{0.33\linewidth}
    \centering
    \includegraphics[scale=0.33]{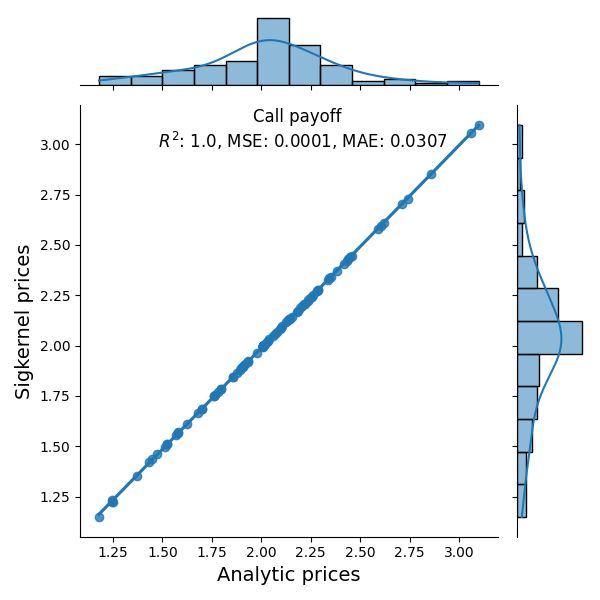}
  \end{minipage}
 \caption{Analytic prices against Signature Kernel prices with $400$ collocation points.}
 \label{fig:fbm}
\end{figure}

As it can be observed from \Cref{fig:fbm}, our kernel approach is capable to recovering with a high level of accuracy the prices for all considered payoff profiles.


\subsection{Rough Bergomi}\label{sec:rBergomi}

In this section, we make use of the following formulation of the rough Bergomi model introduced by \cite{bayer2016pricing} and already written down in the introduction
$$X_t := \int_0^t \sqrt{ \psi_r(\Wh_r) } \mathrm{d}B_r - \frac{1}{2}\int_0^t  \psi_r(\Wh_r) \mathrm{d}r,\quad B_r:=\rho W_r^1 + \sqrt{1 - \rho^2}W_r^2, $$
$$
\psi_t(x) := \xi \exp \left\{ \eta x - \frac{\eta^2}{ 2} t^{2H}\right\}, \quad \Wh^H_t := \sqrt{2H} \int_0^t (t - u)^{H-\half}  \mathrm{d}W^1_r,$$
for independent  Brownian motions $(W^1, W^2)$. The \emph{hybrid scheme} of \cite{bennedsen2017hybrid} is used for efficient, $\mathcal{O}(\ell\log \ell)$, simulation of the Volterra process~$\Wh$ where $\ell$ is the length of the simulated paths. For the experiments we take the same parameters as in the repository\footnote{\url{https://github.com/ryanmccrickerd/rough_bergomi/tree/master}.}, namely $T=1$, $\xi=0.055$, $\eta=1.9$, $\rho =-0.9$. We set the ground truth to be the Monte Carlo prices obtained with $100000$ sample paths. We recall that the  value function~$u$ defined in subsection~\ref{subsec:example_rBergomi} solves the PPDE
\begin{equation}
\left\{
\begin{array}{rl}
     &\partial_t u + \frac{1}{2}\psi_t(\gamma_t) (\partial_{x}^2-\partial_x) u  + \half  \tpar_{\gamma}^{K^t\,K^t} u+ \rho \sqrt{\psi_t(\gamma_t)}  \tpar_{\gamma}^{K^t} (\partial_x u) =0,\\
     &u(T,x,\gamma)=\phi(x).
\end{array}
\right.
\end{equation}
To evaluate the conditional expectation in the Monte Carlo benchmark we remark that, setting $$
I^t_s := \int_t^s K(s,r)\D W_r, \quad\text{where  } K(s,r)=\sqrt{2H}(s-r)^{H-\half}
$$ if $s>t$ and $0$ otherwise, the value function~$u$ has the following probabilistic representation
\begin{equation}
    u(t,x,\gamma) =  \mathbb E \left[f\left( x + \int_t^T \sqrt{\psi_s(\gamma_s + I^t_s)} \D B_s - \half \int_t^T \psi_s(\gamma_s + I^t_s)\ds\right)\right]
\end{equation}
since $\Theta^t$ is an $\cF_t$ measurable process and $I^t$ is independent from $\cF_t$.
Collocation points for the kernel method were sampled uniformly at random on $[0,1]$ for the time variable, uniformly at random on $[x_{min}, x_{max}]$ for the price variables, where $x_{min}, x_{max}$ were chosen using ground truth prices, and sampled from the process $\Theta^t$ for the path variable.


\begin{table}[ht]
\centering
\caption*{$H=0.1$ -- Strike $= 0.1$}
\begin{tabular}{c ccc ccc}
\toprule
\multirow{2}{*}{} 
        Model & \multicolumn{3}{c}{Monte Carlo} & \multicolumn{3}{c}{SigPPDE}   \\
    \cmidrule(lr){2-4} \cmidrule(lr){5-7}
        Complexity & 20 paths  & 100 paths  & 500 paths  & 20 c. pts  & 100 c. pts  & 500 c. pts          \\
    \midrule
\multirow{1}{*}{MSE}  
        & 0.2070 & 0.0079 & 0.0080 & 0.2084 & 0.0016 & 0.0012  \\       
    \addlinespace
\multirow{1}{*}{MAE}
        & 1.3193 & 0.2683 & 0.2300 & 0.6375 & 0.0714 & 0.0672            \\
    \bottomrule
\end{tabular}
\end{table}

\begin{table}[ht]
\centering
\caption*{$H=0.1$ -- Strike $= 1.0$}
\begin{tabular}{c ccc ccc}
\toprule
\multirow{2}{*}{} 
        Model & \multicolumn{3}{c}{Monte Carlo} & \multicolumn{3}{c}{SigPPDE}   \\
    \cmidrule(lr){2-4} \cmidrule(lr){5-7}
        Complexity & 20 paths  & 100 paths  & 500 paths  & 20 c. pts  & 100 c. pts  & 500 c. pts          \\
    \midrule
\multirow{1}{*}{MSE}  
        & 0.0324 & 0.0013  & 0.0001  & 0.1576 & 0.0092  & 0.0008  \\       
    \addlinespace
\multirow{1}{*}{MAE}
        & 0.0679 & 0.0557 & 0.0131 & 0.7998  & 0.2671  & 0.0938           \\
    \bottomrule
\end{tabular}
\end{table}

\begin{table}[ht]
\centering
\caption*{$H=0.3$ -- Strike $= 0.1$}
\begin{tabular}{c ccc ccc}
\toprule
\multirow{2}{*}{} 
        Model & \multicolumn{3}{c}{Monte Carlo} & \multicolumn{3}{c}{SigPPDE}   \\
    \cmidrule(lr){2-4} \cmidrule(lr){5-7}
        Complexity & 20 paths  & 100 paths  & 500 paths  & 20 c. pts  & 100 c. pts  & 500 c. pts          \\
    \midrule
\multirow{1}{*}{MSE}  
        & 0.0198 & 0.0013 & 0.0005  & 0.1198 & 0.0026 & 0.0002  \\       
    \addlinespace
\multirow{1}{*}{MAE}
        & 0.3330 & 0.0718 & 0.0624 & 0.6344  & 0.0973  & 0.0284           \\
    \bottomrule
\end{tabular}
\end{table}

\begin{table}[ht]
\centering
\caption*{$H=0.3$ -- Strike $= 1.0$}
\begin{tabular}{c ccc ccc}
\toprule
\multirow{2}{*}{} 
        Model & \multicolumn{3}{c}{Monte Carlo} & \multicolumn{3}{c}{SigPPDE}   \\
    \cmidrule(lr){2-4} \cmidrule(lr){5-7}
        Complexity & 20 paths  & 100 paths  & 500 paths  & 20 c. pts  & 100 c. pts  & 500 c. pts          \\
    \midrule
\multirow{1}{*}{MSE}  
        & 0.0113 & 0.0015 & 0.0006  & 0.4201 & 0.0046 & 0.0014  \\       
    \addlinespace
\multirow{1}{*}{MAE}
        & 0.2341 & 0.0945 & 0.0679 & 1.1832  & 0.1050  & 0.0905           \\
    \bottomrule
\end{tabular}
\end{table}



We consider two values of $H=0.1$ and $H=0.3$ and two strike values $K=0.1$ and $K=1.0$. We take the same number of sample paths and collocation points for Monte Carlo and the signature kernel solver respectively, and consider three different increasing values, namely $20$, $100$ and $500$. As it can be observed from the tables, the numerical performance in terms of MSE and MAE of the two algorithms are comparable across the different settings. As shown in the tables above, the relative performance depends on the roughness of the input paths (governed by the Hurst parameter $H$) and on the strike of the option. Hence, the comparison is not one-sided. 

The time complexity of a Monte Carlo scheme scales roughly as $\mathcal{O}(h^{-2} N d^2)$, where $N$ is the number of sample paths, $h$ is the discretisation step and $d$ is the dimension of the process. This complexity can further reduced to $\mathcal{O}((h\log(h))^{-1} N d^2)$ using schemes such as the one proposed by \cite{bennedsen2017hybrid}. The complexity for evaluating our solver (once it has been trained offline) scales as $\mathcal{O}(h^{-2} \hat{N}d^2)$ where $\hat{N}$ is the number of collocation points. This cost can be further reduced to $\mathcal{O}(h^{-1} \hat{N} d^2)$ if the operations are carried out on a GPU, see \cite{salvi2021signature} for additional details. It is clear that a transparent comparison of complexities between our approach and Monte Carlo for pricing under rough volatility can only be achieved by obtaining error rates for the convergence of the approximations governing the values of $N$ and $\hat N$ respectively in the above complexities, which we intend to explore as future work.

Nevertheless, the PDE approach provides by design advantages over Monte Carlo methods. Indeed, finite-dimension greeks are prone to instability, let alone pathwise ones, while all derivatives of the PPDE solution appear from (the linear combination of) standard PDEs.

\subsection{Comparison with neural PDE solvers}\label{sec:nn_comparison}

As anticipated, we will conclude this section with some remarks aimed at comparing our kernel approach to the numerous neural network techniques to solve PPDE which have been proposed in the literature in recent years. Neural PDE solvers such as DGM \cite{sirignano2018dgm} and PINNs \cite{raissi2019physics} essentially parameterise the solution of the PDE as a neural network which is then trained using the PDE and its boundary conditions as loss function. One drawback of these approaches in the path-dependent setting \cite{sabate2020solving, saporito2021path, jacquier2023deep} is that they require a time-discretisation of the solution. The time grids over which solutions and derivatives are evaluated during training and testing need to agree, usually followed by an arbitrary interpolation. On the contrary, our kernel approach operates directly on continuous paths, and it is therefore mesh-free. Furthermore, derivatives can be accessed directly by differentiating kernels, with no need to prematurely discretise in time. We note that concomitant to this paper is \cite{fang2023neural}, which uses a \emph{neural rough differential equation} (Neural RDE) model \cite{morrill2021neural, salvi2022neural} to parameterise the solution of a PPDE and is therefore also mesh-free. Another drawback of neural PDE solvers is that their theoretical analysis is often limited to density or universal approximation results; showing the existence of a network of a requisite size achieving a certain error rate, without guarantees whether this network is computable in practice. In fact the optimisation, done with variants of stochastic gradient descent, is never convex so there is no guarantee to attain a global minimum. Our kernel approach boils down to a convex finite-dimensional optimisation,  guaranteed to attain the unique global minimum, and that is consistent with the original problem when the number of collocation points is sent to infinity.


\section{Outlook}\label{sec:outlook}
The numerical solver presented in this paper competes with Monte Carlo methods in terms of the balance between complexity and accuracy. Our next objective will be to derive error estimates with respect to the number of collocation points in order to carry out a fair comparison. This step is essential for choosing the collocation points in an optimal way and better assess the efficiency of the algorithm. For this purpose we may follow the approach of~\cite{batlle2023error} or decide to regularise our optimisation problem as in~\cite{christmann2007consistency}. 
Further extensions include PPDEs arising from general Volterra processes, non-linear PPDEs, and path-dependent payoffs as the VIX example presented in Subsection~\ref{subsubsec:VIX}. Finally, the inverse problem is particularly relevant for financial applications. We expect that the calibration of model parameters can be considerably sped up if the kernel is enhanced with this additional data.



 
\appendix

\section{Proofs of signature and signature kernel estimates}\label{sec:appendix}
\label{appendix:Proofs}


\subsection{Differentiability and continuity estimates for the signature}\label{sec:ProofLemmaSig}
For the sake of uncluttering notations, in this section we fix~$t\in\bT$ and we will write~$\abs{\gamma}_1$ instead of~$\abs{\gamma}_{1-var;[t,T]}$ and~$\norm{\gamma}_0$ instead of~$\norm{\gamma}_{0;[t,T]}$ for every $\gamma\in\BV_t$. 
We present estimates for the signature and its derivatives, in passing obtaining CDEs satisfied by the latter, and showing continuity with respect to the $\cC_0$-norm as claimed in Proposition~\ref{prop:cty_signature}.
The following lemma is inspiredfrom~\cite{friz2010multidimensional}, Theorems 3.15 and 4.4.
\begin{lemma}\label{lemma:bound_sig_appendix}
Let Assumption~\ref{assu:statickernel} hold for~$g$.
    For all~$\gamma,\gammab,\eta,\etab\in \BV_t$, the signature is two times differentiable. Morever, there exist continuous functions $\overline{S_k}:\bR_+^k\to\bR_+$ and~$\overline{DS_k}:\bR_+^k\to\bR_+$, for~$k=1,2,3$, such that the following estimates hold
    \begin{alignat}{1}
        &\sup_{t'\in[t,T]}\norm{S(\gamma^g)_{t'}}_{\overline{T(\cH_g)}} 
        \le  \overline{S_0}(\abs{\gamma}_1); \label{eq:bound_sig_S0}\\
        &\sup_{t'\in[t,T]}\norm{ S(\gamma^g)_{t'}-S(\bar{\gamma}^g)_{t'}}_{\overline{T(\cH_g)}} 
\le \norm{\gamma-\bar{\gamma}}_{0} \overline{DS_0}(\abs{\gamma}_{1}, \abs{\bar\gamma}_{1}); \label{eq:bound_sig_DS0}\\
       & \sup_{t'\in[t,T]}
\norm{\partial^\eta_\gamma S(\gamma^g)_{t'}}_{\overline{T(\cH_g)}} 
\le  \overline{S_1}(\abs{\gamma}_1,\norm{\eta}_1); \label{eq:bound_sig_S1}\\
& \sup_{t'\in[t,T]}\norm{\partial^\eta_\gamma S(\gamma^g)_{t'}-\partial^\eta_\gamma S(\bar{\gamma}^g)_{t'}}_{\overline{T(\cH_g)}} 
\le \norm{\gamma-\bar{\gamma}}_{0} \overline{DS_1}(\abs{\gamma}_{1}, \abs{\bar\gamma}_{1}, \norm{\eta}_{1}); \label{eq:bound_sig_DS1}\\
&\sup_{t'\in[t,T]}     \norm{\partial^{\eta\etab}_\gamma S(\gamma^g)_{t'} }_{\overline{T(\cH_g)}} 
    \le \overline{S_2}(\abs{\gamma}_1,\norm{\eta}_1,\norm{\etab}_1);\label{eq:bound_sig_S2}\\
    &\sup_{t'\in[t,T]}\norm{\partial^{\eta\etab}_\gamma S(\gamma^g)_{t'}-\partial^{\eta\etab}_\gamma S(\bar{\gamma}^g)_{t'}}_{\overline{T(\cH_g)}} 
\le \norm{\gamma-\bar{\gamma}}_{0} \overline{DS_2}(\abs{\gamma}_{1}, \abs{\bar\gamma}_{1}, \norm{\eta}_{1}, \norm{\etab}_{1}). \label{eq:bound_sig_DS2}
\end{alignat}
In particular, $\partial^\eta_\gamma S\circ G$ and
 $\partial^{\eta\etab}_\gamma S\circ G$ are Fréchet derivatives.
\end{lemma}
\begin{proof}
Let~$0\le t\le t'\le T$ and $\gamma,\gammab,\eta,\etab\in \BV_t$.

\textbf{1)} As in Lemma~\ref{lemma:sig_cde}, the signature on lifted paths satisfies the CDE
\begin{align*}
    S(\gamma^g)_{t'} = \bm{1}+ \int_t^{t'} S(\gamma)_u \cdot \D \gamma_u^g = \bm{1}+ \int_t^{t'}  S(\gamma)_u \cdot \nabla g(\gamma_u,\cdot) \D \gamma_u,
\end{align*}
hence we have the bound
\begin{align}\label{eq:Bound_SigEq}
\norm{S(\gamma^g)_{t'}}_{\overline{T(\cH_g)}} 
&\le
 \bm{1}+ \int_t^{t'} \norm{S(\gamma)_u}_{\overline{T(\cH_g)}}  \cdot \norm{\nabla g(\gamma_u,\cdot)\D \gamma_u}_{\overline{T(\cH_g)}} \\
 &\le\bm{1}+ \int_t^{t'} \norm{S(\gamma)_u}_{\overline{T(\cH_g)}}   \norm{\nabla g(\gamma_u,\cdot)}_{\cH_g} \abs{\D \gamma_u},    \nonumber
\end{align}
where $\norm{\nabla g(\gamma_u,\cdot)}_{\cH_g}=\sum_{i,j=1}^e \partial^2_{x^i y^j} g(\gamma_u,\gamma_u) \le  \overline{g_{11}}$ and~$\abs{\D\gamma_u}=\abs{\dot\gamma_u}\du$. By Grönwall's lemma \cite[Lemma 3.2]{friz2010multidimensional} this yields
$$
\sup_{t'\in[t,T]}\norm{S(\gamma^g)_{t'}}_{\overline{T(\cH_g)}} \le \E^{  \overline{g_{11}} \abs{\gamma}_{1}} =: \overline{S_0}(\gamma).
$$
\textbf{2)} Furthermore, \cite[Proposition 2.9]{friz2010multidimensional} gives the $1$-variation bound
$$
\abs{\norm{S(\gamma^g)_{\cdot}}_{\overline{T(\cH_g)}}}_{1}
\le \overline{S_0}(\gamma)  \overline{g_{11}} \abs{\gamma}_{1}.
$$
For any~$\bar{\gamma}\in\BV_t$, imitating the proof of~\cite[Theorem 3.15]{friz2010multidimensional} yields
\begin{align*}
    \norm{S(\gamma^g)_{t'}-S(\bar{\gamma}^g)_{t'}}_{\overline{T(\cH_g)}} 
&\le \int_t^{t'} \norm{S(\gamma^g)_{u}-S(\bar{\gamma}^g)_{u}}_{\overline{T(\cH_g)}} \abs{\D \gamma^g_u} + \norm{ \int_t^{t'} S(\gammab^g)\D(\gamma^g-\gammab^g)_u}_{\overline{T(\cH_g)}} \\
&\le \int_t^{t'} \norm{S(\gamma^g)_{u}-S(\bar{\gamma}^g)_{u}}_{\overline{T(\cH_g)}}   \overline{g_{11}}\abs{\D \gamma_u} + \norm{\norm{\gamma^g-\gammab^g}_{\cH_g}}_0 \left(1+\abs{\norm{S(\gammab^g)_{\cdot}}_{\overline{T(\cH_g)}}}_{1}\right)
\end{align*}
Thus, applying the inequality~$\norm{\norm{\gamma^g-\gammab^g}_{\cH_g}}_0 \le \overline{g_{11}} \norm{\gamma-\gammab}_0 $ and Grönwall's lemma we obtain
$$
\sup_{t'\in[t,T]}\norm{S(\gamma^g)_{t'}-S(\bar{\gamma}^g)_{t'}}_{\overline{T(\cH_g)}} 
\le \norm{\gamma-\bar{\gamma}}_{0}\overline{g_{11}}\big(1+\overline{S_0}(\gammab)  \overline{g_{11}} \abs{\gammab}_{1}\big)
\E^{ \overline{g_{11}} \abs{\gamma}_{1}}
=:  \norm{\gamma-\bar{\gamma}}_{0} \overline{DS_0}(\abs{\gamma}_1,\abs{\gammab}_1),
$$
where~$\overline{DS_0}:\bR^2\to\bR$ is a continuous function.

\textbf{3)} 
We continue with the following observation. For all~$u\in[t,T]$ we have
\begin{align}
     \lim_{\ep \to 0}\frac{1}{\ep} \D (\gamma + \ep \eta)^g_u - \gamma^g_u)
    &= \lim_{\ep \to 0}\frac{1}{\ep} 
    \sum_{i=1}^e \Big\{\partial_i g(\gamma_u+\ep \eta_u,\cdot)\D(\gamma_u^i+\ep\eta_u^i) - \partial_i g(\gamma_u,\cdot)\D\gamma_u^i\Big\} \nonumber\\
    &=\lim_{\ep \to 0}\frac{1}{\ep} 
    \sum_{i=1}^e  \big\{\partial_i g(\gamma_u+\ep \eta_u,\cdot)- \partial_i g(\gamma_u,\cdot) \big\}\D\gamma_u^i + \lim_{\ep \to 0}
    \sum_{i=1}^e \partial_i g(\gamma_u+\ep\eta_u,\cdot)\D\eta^i_u \nonumber\\
    &= \sum_{i=1}^e \bigg\{ \partial^\eta_\gamma \partial_{x^i} g(\gamma_u,\cdot) \D \gamma^i_u + \partial_{x^i}  g(\gamma_u,\cdot)  \D \eta_u^i \bigg\} \nonumber\\
    &= \D \left\{ \sum_{i=1}^e\partial_{x^i} g(\gamma_u,\cdot) \eta^i_u\right\} \nonumber \\
    &= \D \Big\{ \partial^\eta_\gamma G(\gamma)_u\Big\}, \label{eqn:dgamma_dG}
\end{align}
where we recall that~$G(\gamma)_t=g(\gamma_t,\cdot)$.
Further, we note that, for all~$\ep\in(0,1)$, 
$$
\frac{1}{\ep}\sup_{t'\in[t,T]}\norm{S(\gamma^g)_{t'}-S((\gamma+\ep\eta)^g)_{t'}}_{\overline{T(\cH_g)}} 
\le \norm{\eta}_0 \overline{DS_0}(\abs{\gamma}_1,\abs{\gamma}_1+\abs{\eta}_1).
$$
Hence, by dominated convergence and \eqref{eqn:dgamma_dG}, we get
\begin{align}
    \partial^\eta_\gamma S(\gamma^g)_{t'} &=\lim_{\ep\to0} \frac{1}{\ep} \big(S((\gamma+\ep\eta)^g)_{t'}-S(\gamma^g)_{t'} \big) \nonumber\\
    & = \int_t^{t'} \lim_{\ep\to0} \frac{1}{\ep}\big(S((\gamma+\ep\eta)^g)_u-S(\gamma^g)_u\big) \cdot \D\gamma_u^g  + \int_t^{t'} \lim_{\ep\to0} \frac{1}{\ep} S(\gamma^g)_u \cdot \D ((\gamma+\ep\eta)^g)_u-\gamma^g_u)\nonumber \\
    &= \int_t^{t'} \partial^\eta_\gamma S(\gamma^g)_u \cdot  \D\gamma^g_u
    + \int_t^{t'}  S(\gamma^g)_u \cdot \D \left\{\partial^\eta_\gamma G(\gamma)_u\right\}.
    \label{eq:PDE_dersig}
\end{align}
Note that $\norm{\partial_{x^{ij}}^2g(\gamma_u,\cdot)}_{\cH_g} = \partial_{x^{ij}y^{ij}}^4g(\gamma_u,\gamma_u)$, and hence 
\begin{align}
\int_t^{t'}  \norm{\D \left\{\partial^\eta_\gamma G(\gamma)_u\right\} }_{\cH_g}
&\le 
\int_t^{t'} \sum_{i,j=1}^e \norm{\partial_{x^{ij}}g(\gamma_u,\cdot)}_{\cH_g} \abs{\eta_u} \abs{\D \gamma_u} + \int_t^{t'} \sum_{i=1}^e \norm{\partial_{x^{i}}g(\gamma_u,\cdot)}_{\cH_g} \abs{\D \eta_u}  \nonumber\\
& \le \overline{g_{22}} \norm{\eta}_0 \abs{\gamma}_{1} + \overline{g_{11}} \abs{\eta}_{1}\le\overline{G_1}(\abs{\gamma}_1,\norm{\eta}_1), 
\label{eq:BG1}
\end{align}
where $\overline{G_1}:\bR_+^2\to\bR_+$ is a continuous function.
Applying Grönwall's lemma as in the previous steps yields the estimate
$$
\norm{\partial^\eta_\gamma S(\gamma^g)_{t'}}_{\overline{T(\cH_g)}} 
\le \overline{S_0}(\abs{\gamma}_1)  \overline{G_1}(\abs{\gamma}_1,\norm{\eta}_1) \E^{ \overline{g_{11}} \abs{\gamma}_{1}}=: \overline{S_1}(\abs{\gamma}_1,\norm{\eta}_1).
$$
\textbf{4)} 
Noting that, for all~$t\le s'\le t'\le T$,
$$
\abs{\norm{\int_{t}^{\cdot}  S(\gamma^g)_u \cdot \D \left\{\partial^\eta_\gamma G(\gamma)_u\right\}}_{\overline{T(\cH_g)}}}_{1}
\le \overline{S_0}(\gamma) \abs{ \int_t^{\cdot}  \norm{\D \left\{\partial^\eta_\gamma G(\gamma)_u\right\} }_{\cH_g}}_1
\le \overline{S_0}(\gamma)  \overline{G_1}(\abs{\gamma}_1,\norm{\eta}_1),
$$
Proposition 2.9 of \cite{friz2010multidimensional} thus gives
\begin{align*}
\abs{\norm{\partial^\eta_\gamma S(\gamma^g)_{t'}}_{\overline{T(\cH_g)}} }_{1}
\le  \overline{S_1}(\abs{\gamma}_1,\norm{\eta}_1) \overline{g_{11}} \abs{\gamma}_{1} + \overline{S_0}(\gamma)\overline{G_1}(\abs{\gamma}_1,\norm{\eta}_1).
\end{align*}
Then, from the integration by parts formula we observe that
\begin{align*}
    \int_t^{t'} \abs{\eta_u}  \D (\abs{\gamma_u-\bar{\gamma}_u})
    = \int_t^{t'}  \abs{\gamma_u-\bar{\gamma}_u} \D \abs{\eta_u} + \abs{\eta_t} \abs{\gamma_t-\bar{\gamma}_t}
    \le \norm{\gamma-\bar{\gamma}}_0\abs{\eta}_{1} + \norm{\eta}_0 \norm{\gamma-\bar{\gamma}}_0.
\end{align*} 
This allows us to study the following regularity, leveraging the uniform bounds of~$g$'s derivatives:
\begin{align*}
    &\int_t^{t'} \norm{\D \left\{\partial^\eta_\gamma (G(\gamma)_u-G(\bar\gamma)_u)\right\} }_{\cH_g}
    \le \int_t^{t'} \sum_{i,j=1}^e \norm{\partial_{x^{ij}}(g(\gamma_u,\cdot)-g(\bar\gamma_u,\cdot))}_{\cH_g} \abs{\eta_u} \abs{\D \gamma_u}\\
    &+ \int_t^{t'} \sum_{i,j=1}^e \norm{\partial_{x^{ij}} g(\bar\gamma_u,\cdot)}_{\cH_g} \abs{\eta_u} \abs{\D (\gamma_u-\bar\gamma_u)}
    +\int_t^{t'} 
    \sum_{i=1}^e \norm{\partial_{x^{i}}(g(\gamma_u,\cdot)-g(\bar\gamma_u,\cdot))}_{\cH_g} \abs{\D \eta_u}  \\
    &\le \overline{g_{33}} \norm{\gamma-\bar{\gamma}}_0 \norm{\eta}_0 \abs{\gamma}_{1}+ \overline{g_{22}}\int_t^{t'} \abs{\eta_u}  \D (\abs{\gamma_u-\bar{\gamma}_u})
    + \overline{g_{22}} \norm{\gamma-\bar{\gamma}}_0 \abs{\eta}_{1}\\
    &\le \norm{\gamma-\gammab}_0 \overline{DG_1}(\abs{\gamma}_1,\abs{\gammab}_1,\norm{\eta}_1),
\end{align*} 
where $\overline{DG_1}:\bR^3_+\to\bR_+$ is a continuous function.
We deduce that the difference of the source terms in~\eqref{eq:PDE_dersig} has the following bound
\begin{align}\label{eq:sourceterm_dersig}
&\norm{\int_t^{t'}  S(\gamma^g)_u \cdot \D \left\{\partial^\eta_\gamma G(\gamma)_u\right\}
-\int_t^{t'}  S(\bar\gamma^g)_u \cdot \D \left\{\partial^\eta_\gamma G(\bar\gamma)_u\right\}}_{\overline{T(\cH_g)}} \\
&\le \sup_{t'\in[t,T]}\norm{S(\gamma^g)_{t'}-S(\bar{\gamma}^g)_{t'}}_{\overline{T(\cH_g)}} \overline{G_1}(\gamma,\eta) 
+  \sup_{t'\in[t,T]}\norm{S(\gamma^g)_{t'}}_{\overline{T(\cH_g)}}  \int_t^{t'}  \norm{\D \left\{\partial^\eta_\gamma (G(\gamma)_u-G(\bar\gamma)_u)\right\} }_{\cH_g}\nonumber\\
&\le \norm{\gamma-\bar{\gamma}}_0 \Big(\overline{DS_0}(\abs{\gamma}_{1}, \abs{\bar\gamma}_{1})\overline{G_1}(\abs{\gamma}_1,\norm{\eta}_1) +\overline{S_0}(\gamma)\overline{DG_1}(\abs{\gamma}_1,\abs{\gammab}_1,\norm{\eta}_1) \Big).\nonumber
\end{align}
In virtue of \eqref{eq:PDE_dersig} and \eqref{eq:sourceterm_dersig}, the same arguments as in the proof of \cite[ Theorem 3.15]{friz2010multidimensional} entails 
 there exists a continuous function~$\overline{DS_1}:\bR^3_+\to\bR_+$ such that
\begin{align*}
    \sup_{t'\in[t,T]}&\norm{\partial^\eta_\gamma S(\gamma^g)_{t'}-\partial^\eta_\gamma S(\bar{\gamma}^g)_{t'}}_{\overline{T(\cH_g)}}\\
    &\le \E^{2 \overline{g_{11}}(\abs{\gamma}_1\vee\abs{\gammab}_1)}\norm{\int_t^{t'}  S(\gamma^g)_u \cdot \D \left\{\partial^\eta_\gamma G(\gamma)_u\right\}
-\int_t^{t'}  S(\bar\gamma^g)_u \cdot \D \left\{\partial^\eta_\gamma G(\bar\gamma)_u\right\}}_{\overline{T(\cH_g)}} \\
&\le \norm{\gamma-\bar{\gamma}}_{0} \overline{DS_1}(\abs{\gamma}_{1}, \abs{\bar\gamma}_{1}, \norm{\eta}_{1}).
\end{align*}
\textbf{5)} 
The latest estimate allows to use dominated convergence one more time and compute the second derivative in the direction of~$\eta,\etab$:
\begin{align}\label{eq:PDE_dertwosig}
    \partial^{\eta\etab}_\gamma S(\gamma^g)_{t'} 
    &= \int_t^{t'} \partial^{\eta\etab}_\gamma S(\gamma^g)_u \D \gamma^g_u 
    + {\rm Source}(\gamma,\eta,\etab),
\end{align}
where 
$$
{\rm Source}(\gamma,\eta,\etab):=\int_t^{t'} \partial^\eta_\gamma S(\gamma^g)_u \cdot  \D \left\{\partial^\etab_\gamma G(\gamma)_u\right\}   +\int_t^{t'} \partial^\etab_\gamma S(\gamma^g)_u \cdot  \D \left\{\partial^\eta_\gamma G(\gamma)_u\right\} 
    + \int_t^{t'}  S(\gamma^g)_u \cdot  \D \left\{\partial^{\eta\etab}_\gamma G(\gamma)_u\right\}.
$$
Using the same arguments as in \eqref{eqn:dgamma_dG} one deduces that
\begin{align}
    \int_t^{t'}  \norm{\D \left\{\partial^{\eta\etab}_\gamma G(\gamma)_u\right\} }_{\cH_g}
&\le \int_t^{t'}\sum_{i,j,k=1}^e \norm{\partial^3_{x^{ijk}} g(\gamma_u,\cdot)}_{\cH_g} \abs{\etab_u}\abs{\eta_u}\abs{\D\gamma_u} \nonumber\\
& \quad+ \int_t^{t'} \sum_{i,j=1}^e \norm{\partial^2_{x^{ij}} g(\gamma_u,\cdot)}_{\cH_g} \Big( \abs{\eta_u} \abs{\D \etab_u} + \abs{\etab_u} \abs{\D \eta_u}\Big) \nonumber\\
& \le  \overline{g_{33}} \norm{\eta}_0 \norm{\etab}_0 \abs{\gamma}_{1} +\overline{g_{22}}\Big(\norm{\eta}_0 \abs{\etab}_{1} + \norm{\etab}_0 \abs{\eta}_{1} 
 \Big)
\le \overline{G_2}(\abs{\gamma}_1,\norm{\eta}_1,\norm{\etab}_1),
\label{eq:BG2}
\end{align}
where $\overline{G_2}:\bR_+^3\to\bR_+$ is a continuous function.
Grönwall's inequality once again yields the bound
\begin{align*}
    &\sup_{t'\in[t,T]} \norm{\partial^{\eta\etab}_\gamma S(\gamma^g)_{t'}}_{\overline{T(\cH_g)}} \\
    &\le \Big(\overline{S_1}(\abs{\gamma}_1,\norm{\eta}_1) \overline{G_1}(\abs{\gamma}_1,\norm{\etab}_1) +  \overline{S_1}(\abs{\gamma}_1,\norm{\etab}_1) \overline{G_1}(\abs{\gamma}_1,\norm{\eta}_1)+\overline{S_0}(\abs{\gamma}_1) \overline{G_2}(\abs{\gamma}_1,\norm{\eta}_1,\norm{\etab}_1)\Big) \E^{\overline{g_{11}} \abs{\gamma}_{1}}\\
    &=:\overline{S_2}(\abs{\gamma}_1,\norm{\eta}_1,\norm{\etab}_1).
\end{align*}
\textbf{6)} The same steps as for the first derivative yield the existence of a continuous function $\overline{DG_2}:\bR^4\to\bR$ such that
$$
 \int_t^{t'}  \norm{\D \left\{\partial^{\eta\etab}_\gamma \big(G(\gamma)_u-G(\gammab)_u\big)\right\} }_{\cH_g} \le \norm{\gamma-\gammab}_0 \overline{DG_2}(\abs{\gamma}_1,\abs{\gammab}_1,\norm{\eta}_1,\norm{\etab}_1).
$$
Hence the difference of the source term for two paths $\gamma,\gammab$ can be bounded as follows
\begin{align*}
&\norm{{\rm Source}(\gamma,\eta,\etab)-{\rm Source}(\gammab,\eta,\etab) }_{\overline{T(\cH_g)}}\\ 
&\le 
\norm{\gamma-\gammab}_0 \Big( \overline{S_1}(\abs{\gamma}_1,\norm{\eta}_1)\overline{DG_1}(\abs{\gamma}_1,\abs{\gammab}_1,\norm{\eta}_1) + \overline{G_1}(\abs{\gamma}_1,\abs{\gammab}_1,\norm{\eta}_1)\overline{DS_1}(\abs{\gamma}_1,\abs{\gammab}_1,\norm{\eta}_1)\\
&\quad+ \overline{S_1}(\abs{\gamma}_1,\norm{\etab}_1)\overline{DG_1}(\abs{\gamma}_1,\abs{\gammab}_1,\norm{\etab}_1) + \overline{G_1}(\abs{\gamma}_1,\abs{\gammab}_1,\norm{\etab}_1)\overline{DS_1}(\abs{\gamma}_1,\abs{\gammab}_1,\norm{\etab}_1) \\
&\quad+ \overline{S_0}(\abs{\gamma}_1)\overline{G_2}(\abs{\gamma}_1,\abs{\gammab}_1,\norm{\eta}_1,\norm{\etab}_1)+ \overline{DS_0}(\abs{\gamma}_1,\abs{\gammab}_1) \overline{G_2}(\abs{\gamma}_1,\abs{\gammab}_1,\norm{\eta}_1,\norm{\etab}_1) \Big).
\end{align*}
Finally, by \cite[Theorem 3.15]{friz2010multidimensional} again there is a continuous function
$\overline{DS_2}:\bR^4\to\bR$ such that
$$
\sup_{t'\in[t,T]}\norm{\partial^{\eta\etab}_\gamma S(\gamma^g)_{t'}-\partial^{\eta\etab}_\gamma S(\bar{\gamma}^g)_{t'}}_{\overline{T(\cH_g)}} 
\le \norm{\gamma-\bar{\gamma}}_{0} \overline{DS_2}(\abs{\gamma}_{1}, \abs{\bar\gamma}_{1}, \norm{\eta}_{1},\norm{\etab}_1)
$$
thus concluding the proof.
\end{proof}

\subsection{Differentiability and continuity estimates for the signature kernel}\label{sec:proofs_kernel}
We start by stating estimates for the signature kernel and its derivatives, which follow from Lemma~\ref{lemma:bound_sig_appendix}. The PDE formulations in~\Cref{prop:der_kappag} are immediate corollaries. In terms of notations, for two paths~$\gamma\in\BV_t$ and~$\tau\in\BV_s$, we will denote~$\abs{\gamma}_1=\abs{\gamma}_{1-var;[t,T]}$ and~$\abs{\tau}_1=\abs{\tau}_{1-var;[s,T]}$.
\begin{lemma}\label{lemma:bound_kappag_appendix} 
Let Assumption and~\ref{assu:statickernel} hold for~$g$. Let~$s,t\in\bT$, the paths $\gamma,\eta,\etab\in \BV_t$ and~$\tau\in\BV_s$. Then the following estimates hold:
\begin{align*}
    &\sup_{s'\in[s,T],t'\in[t,T]} \abs{\kappa^g(\gamma,\tau)_{s',t'}} \le \overline{S_0}(\abs{\gamma}_1) \overline{S_0}(\abs{\tau}_1);\\
    &\sup_{s'\in[s,T],t'\in[t,T]} \abs{\partial^\eta_\gamma \kappa^g(\gamma,\tau)_{s',t'}} \le \overline{S_1}(\abs{\gamma}_1,\norm{\eta}_1) \overline{S_0}(\abs{\tau}_1);\\
    &\sup_{s'\in[s,T],t'\in[t,T]} \abs{\partial^{\eta\etab}_\gamma \kappa^g(\gamma,\tau)_{s',t'}} \le \overline{S_2}(\abs{\gamma}_1,\norm{\eta}_1,\norm{\etab}_1) \overline{S_0}(\abs{\tau}_1).
\end{align*}
\end{lemma}

\begin{proof}
    Let~$0\le s\le s'\le T$, $0\le t\le t'\le T$ and~$\gamma,\eta,\etab\in\BV_t,\tau\in\BV_s$. By the definition of the signature kernel, Cauchy-Schwarz inequality and~\eqref{eq:bound_sig_S0} we have
    $$
\abs{\kappa^g(\gamma,\tau)_{s',t'}}
=\abs{\langle S(\gamma^g)_{t'}, S(\tau^g)_{s'}\rangle}
\le \norm{S(\gamma^g)_{t'}}_{\overline{T(\cH_g)}} \norm{S(\tau^g)_{s'}}_{\overline{T(\cH_g)}}
\le \overline{S_0}(\abs{\gamma}_1) \overline{S_0}(\abs{\tau}_1).
    $$
    Then for any~$\ep>0$, Cauchy-Schwarz inequality and~\eqref{eq:bound_sig_DS0} yield
    \begin{align*}
       \abs{1}{\ep} \abs{\kappa^g(\gamma+\ep\eta,\tau)_{s',t'}-\kappa^g(\gamma,\tau)_{s',t'}} \le \norm{\eta}_0 \overline{DS_0}(\abs{\gamma}_1,\abs{\gamma}_1+\abs{\eta}_1) \overline{S_0}(\abs{\tau}_1).
    \end{align*}
    This dominated convergence argument combined with Cauchy-Schwarz inequality and~\eqref{eq:bound_sig_S1} entail
    \begin{align*}
        \abs{\partial_\gamma^\eta \kappa^g(\gamma,\tau)_{s',t'}}
    = \abs{\langle \partial^\gamma_\eta S(\gamma^g)_{t'}, S(\tau^g)_{s'} \rangle} 
    \le \overline{S_1}(\abs{\gamma}_1,\norm{\eta}_1) \overline{S_0}(\abs{\tau}_1).
    \end{align*}
    The same ideas coupled with the bounds \eqref{eq:bound_sig_S2} and~\eqref{eq:bound_sig_DS2} lead to the last estimate.
\end{proof}

\begin{proof}[Proof of Proposition \ref{prop:der_kappag}]

Let~$0\le s\le s'\le T$ and~$0\le t\le t'\le T$ and~$\gamma,\eta,\etab\in\BV_t,\tau\in\BV_s$.

\textbf{ 1)}
In virtue of \eqref{eq:bound_sig_DS0}, dominated convergence allows to 
 push the limit inside the integrals:
\begin{align*}
    \partial_\gamma^\eta \kappa^g(\gamma,\tau)_{s',t'} &= \lim_{\ep \to 0}\frac{1}{\ep} \left(\kappa^g(\gamma + \ep \eta,\tau)_{s',t'} - \kappa^g(\gamma,\tau)_{s',t'}\right) \\
    &= \lim_{\ep \to 0}\frac{1}{\ep} \int_s^{s'}\int_t^{t'} \left(\kappa^g(\gamma + \ep \eta,\tau)_{u,v} \left\langle \D (\gamma + \ep \eta)^g_u, \D \tau^g_v \right\rangle_{\cH_g} -\kappa^g(\gamma,\tau)_{u,v}\left\langle \D \gamma^g_u, \D \tau^g_v \right\rangle_{\cH_g} \right)\\
    & =\int_s^{s'}\int_t^{t'}  \lim_{\ep \to 0}\frac{1}{\ep}\Big(\kappa^g(\gamma + \ep \eta,\tau)_{u,v}- \kappa^g(\gamma,\tau)_{u,v}\Big) \left\langle\D (\gamma + \ep \eta)^g_u, \D \tau^g_v \right\rangle_{\cH_g} \\
    &\qquad + \kappa^g(\gamma,\tau)_{u,v} \lim_{\ep \to 0}\frac{1}{\ep}\left\langle\D(\gamma + \ep \eta)^g_u - \D\gamma^g_u , \D \tau^g_v \right\rangle_{\cH_g} \\
    &= \int_s^{s'}\int_t^{t'} \left(\partial_\gamma^\eta\kappa^g(\gamma,\tau)_{u,v} \left\langle \D \gamma^g_u, \D \tau^g_v \right\rangle_{\cH_g} + \kappa^g(\gamma,\tau)_{u,v} \lim_{\ep \to 0}\frac{1}{\ep}\left\langle\D(\gamma + \ep \eta)^g_u - \D\gamma^g_u , \D \tau^g_v \right\rangle_{\cH_g} \right).
\end{align*}    
For all~$u\in[t,T]$ and~$v\in[s,T]$ we have similarly to~\eqref{eqn:dgamma_dG}
\begin{align}
    & \lim_{\ep \to 0}\frac{1}{\ep} \Big\langle\D (\gamma + \ep \eta)^g_u - \gamma^g_u , \D \tau^g_v \Big\rangle_{\cH_g} \nonumber
    \\
    &= \lim_{\ep \to 0}\frac{1}{\ep} 
    \sum_{i,k=1}^e \Big\langle  \partial_i g(\gamma_u+\ep \eta_u,\cdot)\D(\gamma_u^i+\ep\eta_u^i) - \partial_i g(\gamma_u,\cdot)\D\gamma_u^i ,\partial_k g(\tau_v,\cdot)\D \tau^k_v \Big\rangle\nonumber\\
    &=\lim_{\ep \to 0}\frac{1}{\ep} 
    \sum_{i,k=1}^e \Big\langle \partial_i g(\gamma_u+\ep \eta_u,\cdot)- \partial_i g(\gamma_u,\cdot), \partial_k g(\tau_v,\cdot) \Big\rangle \D\gamma_u^i\D \tau^k_v + \lim_{\ep \to 0}
    \sum_{i,k=1}^e\Big\langle \partial_i g(\gamma_u+\ep\eta_u,\cdot) , \partial_k g(\tau_v,\cdot) \Big\rangle\D\eta^i_u\D\tau^k_v  \nonumber\\
    &=\sum_{i,k=1}^e\lim_{\ep \to 0}\bigg\{\frac{1}{\ep} 
     \partial_{x^i}\partial_{y^k} \Big( g(\gamma_u+\ep\eta_u,\tau_v)-g(\gamma_u,\tau_v)\Big)\D\gamma_u^i\D \tau^k_v
    +\partial_{x^i}\partial_{y^k} g(\gamma_u+\ep\eta_u,\tau_v)\D\eta^i_u\D\tau^k_v\bigg\} \nonumber\\
    &= \sum_{i,k=1}^e \bigg\{ \partial^\eta_\gamma \partial_{x^i y^k}^2 g(\gamma_u,\tau_v) \D \gamma^i_u + \partial_{x^i y^k}^2  g(\gamma_u,\tau_v)  \D \eta_u^i \bigg\}\D \tau^k_v \nonumber\\
    &= \sum_{i,k=1}^e \bigg\{\sum_{j=1}^e\partial_{x^{ij} y^k}^3  g(\gamma_u,\tau_v)\eta_u^j \D \gamma_u^i + \partial_{x^i y^k}^2 g(\gamma_u,\tau_v)  \D \eta_u^i \bigg\} \D \tau^k_v \label{eqn:dgamma_partial}
\end{align}
where we used the reproducing property of derivatives of $g$ \eqref{eq:reproducing_dg} and we can swap partial derivatives because $g(\cdot,x)\in \cC^2(\bR^e,\bR)$ for all $x\in \bR^e$. Unrolling in the other direction one can also obtain
\begin{align}
    \sum_{i,k=1}^e \bigg\{\sum_{j=1}^e \partial_{x^{ij} y^k}^3  g(\gamma_u,\tau_v)\eta_u^j \D \gamma_u^i + \partial_{x^i y^k}^2 g(\gamma_u,\tau_v)  \D \eta_u^i \bigg\} \D \tau^k_v 
    &= \left\langle \D \left\{ \sum_{i=1}^e\partial_{x^i} g(\gamma_u,\cdot) \eta^i_u\right\} ,\D \tau^g_v\right\rangle \nonumber \\
    &= \left\langle\D \Big\{ \partial^\eta_\gamma G(\gamma)_u\Big\},\D \tau^g_v\right\rangle . \label{eqn:dgamma_dG_kernel}
\end{align}
The expressions \eqref{eqn:dgamma_dG_kernel} and \eqref{eqn:dgamma_partial} yield \eqref{eqn:pde_derivative_G} and \eqref{eqn:pde_derivative_partial} respectively.

\textbf{2)} 
We take the terms of \eqref{eqn:pde_derivative_partial} one by one to compute 
\begin{align*}
    \partial_\gamma^{\eta\etab} \kappa^g(\gamma,\tau)_{s',t'} &:=\lim_{\ep\to0}\frac{1}{\ep}\Big(\partial_\gamma^{\eta} \kappa^g(\gamma+\ep\etab,\tau)_{s',t'}-\partial_\gamma^{\eta} \kappa^g(\gamma,\tau)_{s',t'}\Big) \\
    &= \lim_{\ep\to0}\frac{1}{\ep}\Bigg\{\int_s^{s'}\int_t^{t'} \partial_\gamma^\eta\kappa^g(\gamma+\ep\etab,\tau)_{u,v} \sum_{i,k=1}^e \partial_{x^i y^k}^2 g(\gamma_u+\ep\etab_u,\tau_v) \D (\gamma^i_u+\ep\etab^i_u) \D \tau^k_v \\
    &\qquad\qquad\qquad -\int_s^{s'}\int_t^{t'} \partial_\gamma^\eta\kappa^g(\gamma,\tau)_{u,v} \sum_{i,k=1}^e \partial_{x^i y^k}^2 g(\gamma_u,\tau_v) \D \gamma^i_u \D \tau^k_v \Bigg\} \\
    &+ \lim_{\ep\to0}\frac{1}{\ep}\Bigg\{\int_s^{s'}\int_t^{t'}\kappa^g(\gamma+\ep\etab,\tau)_{u,v} \sum_{i,j,k=1}^e \partial_{x^{ij}y^k}^3
    g(\gamma_u+\ep\etab_u,\tau_v) \eta_u^j \D (\gamma_u^i+\ep\etab^i_u)\D\tau^k_v \\
    &\qquad\qquad\qquad -\int_s^{s'}\int_t^{t'}\kappa^g(\gamma,\tau)_{u,v} \sum_{i,j,k=1}^e \partial_{x^{ij}y^k}^3
    g(\gamma_u,\tau_v) \eta_u^j \D \gamma_u^i\D\tau^k_v\Bigg\} \\
    & + \lim_{\ep\to0}\frac{1}{\ep}\Bigg\{\int_s^{s'}\int_t^{t'}\kappa^g(\gamma+\ep\etab,\tau)_{u,v} \sum_{i,k=1}^e \partial_{x^i y^k}^2 g(\gamma_u+\ep\etab_u,\tau_v)\D \eta_u^i  \D \tau^k_v\\
    &\qquad\qquad\qquad -\int_s^{s'}\int_t^{t'}\kappa^g(\gamma,\tau)_{u,v} \sum_{i,k=1}^e \partial_{x^i y^k}^2 g(\gamma_u,\tau_v)\D \eta_u^i  \D \tau^k_v\Bigg\}.
\end{align*}
Thanks to Lemma \ref{lemma:bound_sig_appendix}, we can use dominated convergence to swap limits and integrals. After rearranging terms we obtain
\begin{align}\label{eqn:pde_derivative_partial2}
    \partial_\gamma^{\eta\etab} \kappa^g(\gamma,\tau)_{s',t'} 
    &= \int_s^{s'}\int_t^{t'} \partial_\gamma^{\eta\etab} \kappa^g(\gamma,\tau)_{u,v} \sum_{i,k=1}^e \partial_{x^i y^k}^2 g(\gamma_u,\tau_v) \D \gamma^i_u \D \tau^k_v \\
    &\quad + \int_s^{s'}\int_t^{t'} \partial_\gamma^{\eta} 
    \kappa^g(\gamma,\tau)_{u,v} \sum_{i,k,l=1}^e \partial_{x^{il} y^k}^3 g(\gamma_u,\tau_v) \etab^l_u \D \gamma^i_u \D \tau^k_v \nonumber \\
    &\quad + \int_s^{s'}\int_t^{t'} \partial_\gamma^{\eta} \kappa^g(\gamma,\tau)_{u,v} \sum_{i,k=1}^e \partial_{x^i y^k}^2 g(\gamma_u,\tau_v) \D \etab^i_u \D \tau^k_v \nonumber \\
    &\quad + \int_s^{s'}\int_t^{t'} \partial^{\etab}_\gamma \kappa^g(\gamma,\tau)_{u,v} \sum_{i,j,k=1}^e\partial_{x^{ij} y^k}^3 g(\gamma_u,\tau_v)\eta_u^j \D \gamma_u^i\D\tau^k_v \nonumber \\
    &\quad + \int_s^{s'}\int_t^{t'}\kappa^g(\gamma,\tau)_{u,v} \sum_{i,j,k,l=1}^e\partial_{x^{ijl} y^k}^4 g(\gamma_u,\tau_v) \etab^l_u \eta_u^j \D \gamma_u^i\D\tau^k_v  \nonumber\\
    &\quad + \int_s^{s'}\int_t^{t'}\kappa^g(\gamma,\tau)_{u,v} \sum_{i,j,k=1}^e\partial_{x^{ij} y^k}^3 g(\gamma_u,\tau_v)\eta_u^j \D \etab_u^i\D\tau^k_v  \nonumber\\
    &\quad + \int_s^{s'}\int_t^{t'} \partial^{\etab}_\gamma\kappa^g(\gamma,\tau)_{u,v} \sum_{i,k=1}^e\partial_{x^i y^k}^2 g(\gamma_u,\tau_v)  \D \eta_u^i  \D \tau^k_v \nonumber\\
    &\quad + \int_s^{s'}\int_t^{t'}\kappa^g(\gamma,\tau)_{u,v} \sum_{i,k,l=1}^e\partial_{x^{il} y^k}^3 g(\gamma_u,\tau_v)\etab^l_u \D \eta_u^i  \D \tau^k_v \nonumber
\end{align}
where we also used the property \eqref{eq:reproducing_dg}. We can regroup the terms according to the order of $\kappa^g$, using $\sum_{i,k=1}^e \partial_{x^i y^k}^2 g(\gamma_u,\tau_v) \D \gamma^i_u \D \tau^k_v = \langle \D\gamma^g_u, \D\tau^g_v\rangle_{\cH_g} $, Equation \eqref{eqn:dgamma_dG_kernel} and similarly for the next order:
\begin{equation}
\Big\langle \D \partial^{\eta\etab}_\gamma G(\gamma)_u, \D\tau^g_v\Big\rangle_{\cH_g} 
= \sum_{i,k,l=1}^e \left\{ \sum_{j=1}^e\partial_{x^{ijl} y^k}^4 g(\gamma_u,\tau_v) \etab^l_u \eta_u^j \D \gamma_u^i + \partial_{x^{ij} y^k}^3 g(\gamma_u,\tau_v)\eta_u^j \D \etab_u^i
+\partial_{x^{il} y^k}^3 g(\gamma_u,\tau_v)\etab^l_u \D \eta_u^i \right\} \D \tau^k_v.    
\label{eq:Expression_DDG}
\end{equation}
Eventually this yields the more compact formulation
\begin{align*}
    \partial_\gamma^{\eta\etab} \kappa^g(\gamma,\tau)_{s,t} =& \int_s^{s'}\int_t^{t'} \partial_\gamma^{\eta\etab} \kappa^g(\gamma,\tau)_{u,v} \langle \D\gamma^g_u, \D\tau^g_v\rangle_{\cH_g}
    +\int_s^{s'}\int_t^{t'} \partial_\gamma^{\eta} \kappa^g(\gamma,\tau)_{u,v} \langle \D\partial_\gamma^{\etab} G(\gamma)_u, \D\tau^g_v\rangle_{\cH_g} \\
    &+ \int_s^{s'}\int_t^{t'} \partial_\gamma^{\etab} \kappa^g(\gamma,\tau)_{u,v} \langle \D\partial^\eta_\gamma G(\gamma)_u, \D\tau^g_v\rangle_{\cH_g}
    + \int_s^{s'}\int_t^{t'} \kappa^g(\gamma,\tau)_{u,v} \langle \D\partial^{\eta\etab}_\gamma G(\gamma)_u, \D\tau^g_v\rangle_{\cH_g},
\end{align*}
as claimed.
\end{proof}




\bibliography{references}
\bibliographystyle{abbrv}

\end{document}